\documentclass[final]{siamart1116}

\usepackage{lipsum}
\usepackage{amsfonts}
\usepackage{graphicx}
\usepackage{epstopdf}
\usepackage{algorithmic}
\ifpdf
  \DeclareGraphicsExtensions{.eps,.pdf,.png,.jpg}
\else
  \DeclareGraphicsExtensions{.eps}
\fi

\numberwithin{theorem}{section}

\newcommand{\TheTitle}{Ordering positive definite matrices} 
\newcommand{\TheAuthors}{C. Mostajeran, and R. Sepulchre}

\headers{\TheTitle}{\TheAuthors}

\title{{\TheTitle}\thanks{
This work was funded by the Engineering and Physical Sciences Research Council (EPSRC) of the United Kingdom, as well as the European Research Council under the Advanced ERC Grant Agreement Switchlet n.670645.}}

\author{
  Cyrus Mostajeran\thanks{Department of Engineering, University of Cambridge, United Kingdom (\email{csm54@cam.ac.uk}, \email{r.sepulchre@eng.cam.ac.uk}).}
  \and
  Rodolphe Sepulchre\footnotemark[2]
}

\usepackage{amsopn}

\ifpdf
\hypersetup{
  pdftitle={\TheTitle},
  pdfauthor={\TheAuthors}
}
\fi

\def\tr{\operatorname{tr}}

\newtheorem{remark}{Remark}

\usepackage{enumerate}

\begin{document}

\maketitle

\begin{abstract}
We introduce new partial orders on the set $S^+_n$ of positive definite matrices of dimension $n$ derived from the affine-invariant geometry of $S^+_n$. The orders are induced by affine-invariant cone fields, which arise naturally from a local analysis of the orders that are compatible with the homogeneous geometry of $S^+_n$ defined by the natural transitive action of the general linear group $GL(n)$. We then take a geometric approach to the study of monotone functions on $S^+_n$ and establish a number of relevant results, including an extension of the well-known L\"owner-Heinz theorem derived using differential positivity with respect to affine-invariant cone fields. 
\end{abstract}  



\section{Introduction}

Well-defined notions of ordering of elements of a space are of fundamental importance to many areas of applied mathematics, including the theory of monotone functions and matrix means in which orders play a defining role \cite{Lowner1934,Heinz1951,Ando1979, Kubo1979}. Partial orders play a key part in a wide variety of applications across information geometry where one is interested in performing statistical analysis on sets of matrices. In such applications, the choice of order relation is often taken for granted. This choice, however, is of crucial significance since a function that is not monotone with respect to one order, may be monotone with respect to another.

We outline a geometric approach to systematically generate orders on homogeneous spaces. A homogeneous space is a manifold that admits a transitive action by a Lie group, in the sense that any two points on the manifold can be mapped onto each other by elements of a group of transformations that act on the space. The observation that cone fields induce conal orders on continuous spaces, combined with the geometry of homogeneous spaces forms the basis of the approach taken in this paper. The aim is to generate cone fields that are invariant with respect to the homogeneous geometry, thereby defining partial orders built upon the underlying symmetries of the space. A smooth cone field on a manifold is often also referred to as a causal structure.  The geometry of invariant cone fields and causal structures on homogeneous spaces has been the subject of extensive studies from a Lie theoretic perspective; see \cite{Neeb1991,Hilgert1997,Hilgert1989}, for instance. Causal structures induced by quadratic cone fields on manifolds also play a fundamental role in mathematical physics, in particular within the theory of general relativity \cite{Segal1976}.

The focus of this paper is on ordering the elements of the set of symmetric positive definite matrices $S^+_n$ of dimension $n$. Positive definite matrices arise in numerous applications, including as covariance matrices in statistics and computer vision, as variables in convex and semidefinite programming, as unknowns in fundamental problems in systems and control theory, as kernels in machine learning, and as diffusion tensors in medical imaging. The space $S^+_n$ forms a smooth manifold that can be viewed as a homogeneous space admitting a transitive action by the general linear group $GL(n)$, which endows the space with an affine-invariant geometry as reviewed in Section \ref{section 2}. In Section \ref{section 3}, this geometry is used to construct affine-invariant cone fields and new partial orders on $S^+_n$. In Section \ref{section 4}, we discuss how differential positivity \cite{Forni2015} can be used to study and characterize monotonicity on $S^+_n$ with respect to the invariant orders introduced in this paper. We also state and prove a generalized version of the celebrated L\"owner-Heinz theorem  \cite{Lowner1934,Heinz1951} of operator monotonicity theory derived using this approach. In Section \ref{half spaces}, we consider preorder relations induced by affine-invariant and translation-invariant half-spaces on $S^+_n$, and provide examples of functions and flows that preserve such structures. Finally, in Section \ref{matrix means}, we review the notion of matrix means and establish a connection between the geometric mean and affine-invariant cone fields on $S^+_n$.

\section{Homogeneous geometry of $S^+_n$}  \label{section 2}

The set $S^+_n$ of symmetric positive definite matrices of dimension $n$ has the structure of a homogeneous space with a transitive $GL(n)$-action. 
The transitive action of $GL(n)$ on $S^+_n$ is given by congruence transformations of the form
\begin{equation}
\tau_A: \Sigma \mapsto A\Sigma A^T \quad \forall A\in GL(n), \ \forall \Sigma \in S^+_n.
\end{equation}
Specifically, if $\Sigma_1,\Sigma_2\in S^+_n$, then $\tau_A$ with $A={\Sigma_2^{1/2}\Sigma_1^{-1/2}}\in GL(n)$  maps $\Sigma_1$ onto $\Sigma_2$, where $\Sigma^{1/2}$ denotes the unique positive definite square root of $\Sigma$. This action is said to be \emph{almost effective} in the sense that $\pm I$ are the only elements of $GL(n)$ that fix every $\Sigma \in S^+_n$. The isotropy group of this action at $\Sigma = I$ is precisely the orthogonal group $O(n)$, since $\tau_Q: I \mapsto QIQ^T=I$ if and only if $Q\in O(n)$. Thus, we can identify any $\Sigma\in S^+_n$ with an element of the quotient space $GL(n)/O(n)$. That is
\begin{equation}  \label{quotient}
S^+_n\cong GL(n)/O(n).
\end{equation}
The identification in (\ref{quotient}) can also be made by noting that $\Sigma\in S^+_n$ admits a Cholesky decomposition
$\Sigma = CC^T$ for some $C\in GL(n)$.
The Cauchy polar decomposition of the invertible matrix $C$ yields a unique decomposition $C=PQ$ of $C$ into an orthogonal matrix $Q\in O(n)$ and a symmetric positive definite matrix $P\in S^n_+$. Now note that if $\Sigma$ has Cholesky decomposition $\Sigma = CC^T$ and $C$ has a Cauchy polar decomposition $C=PQ$, then $\Sigma=PQQ^TP=P^2$. That is, $\Sigma$ is invariant with respect to the orthogonal part $Q$ of the polar decomposition. Therefore, we can identify any $\Sigma\in S^+_n$ with the equivalence class $[\Sigma^{1/2}]=\Sigma^{1/2}\cdot O(n)$ in the quotient space $GL(n)/O(n)$.

Recall that the Lie algebra $\mathfrak{gl}(n)$ of $GL(n)$ consists of the set $\mathbb{R}^{n\times n}$ of all real $n\times n$ matrices
equipped with the Lie bracket $[X,Y]=XY-YX$, while the Lie algebra of $O(n)$ is $\mathfrak{o}(n)=\{X\in \mathbb{R}^{n\times n}: X^T=-X\}$. Since any matrix $X\in \mathbb{R}^{n\times n}$ has a unique decomposition 
$
X=\frac{1}{2}(X-X^T) + \frac{1}{2}(X+X^T)$,
as a sum of an antisymmetric part and a symmetric part, we have $\mathfrak{gl}(n)=\mathfrak{o}(n)\oplus\mathfrak{m}$, where $\mathfrak{m}=\{X\in\mathbb{R}^{n\times n}: X^T=X\}$. Furthermore, since $\operatorname{Ad}_Q(S)=QSQ^{-1}=QSQ^T$ is a symmetric matrix for each $S\in\mathfrak{m}$, we have
\begin{equation}
\operatorname{Ad}_{O(n)}\mathfrak{m}\subseteq\mathfrak{m},
\end{equation}
which shows that 
$S^+_n=GL(n)/O(n)$ is in fact a reductive homogeneous space with reductive decomposition $\mathfrak{gl}(n)=\mathfrak{o}(n)\oplus\mathfrak{m}$. Also, note that since $(XY-YX)^T=Y^TX^T-X^TY^T$, we have
$[\mathfrak{o}(n),\mathfrak{o}(n)]\subseteq\mathfrak{o}(n)$, $[\mathfrak{m},\mathfrak{m}]\subseteq \mathfrak{o}(n)$, and $ [\mathfrak{o}(n),\mathfrak{m}]\subseteq\mathfrak{m}$.
The tangent space $T_oS^+_n$ of $S^+_n$ at the base-point $o=[I]=I\cdot O(n)$ is identified with $\mathfrak{m}$. For each $\Sigma \in S^+_n$, the action $\tau_{\Sigma^{1/2}}:S^+_n\to S^+_n$ induces the vector space isomorphism $d\tau_{\Sigma^{1/2}}\vert_I: T_I S^+_n \to T_{\Sigma} S^+_n$ given by
\begin{equation}   \label{isomorphism}
 d\tau_{\Sigma^{1/2}}\big\vert_I X = \Sigma^{1/2}X\Sigma^{1/2}, \quad \forall X\in \mathfrak{m}.
\end{equation}
The map (\ref{isomorphism}) can be used to extend structures defined in $T_o S^+_n$ to structures defined on the tangent bundle $T S^+_n$ through affine-invariance, provided that the structures in $T_o S^+_n$  are $\operatorname{Ad}_{O(n)}$-invariant. The $\operatorname{Ad}_{O(n)}$-invariance is required to ensure that the extension to $T S^+_n$ is unique and thus well-defined. For instance, any homogeneous Riemannian metric on $S^+_n\cong GL(n)/O(n)$ is determined by an $\operatorname{Ad}_{O(n)}$-invariant inner product on $\mathfrak{m}$. Any such inner product induces a norm that is rotationally invariant and so can only depend on the scalar invariants $\operatorname{tr}(X^k)$ where $k \geq 1$ and $X\in\mathfrak{m}$. Moreover, as the inner product is a quadratic function, $\|X\|^2$ must be a linear combination of $(\operatorname{tr}(X))^2$ and $\operatorname{tr}(X^2)$. Thus, any  $\operatorname{Ad}_{O(n)}$-invariant inner product on $\mathfrak{m}$ must be a scalar multiple of 
\begin{equation}
\langle X, Y \rangle_{\mathfrak{m}} = \operatorname{tr}(XY)+\mu \operatorname{tr}(X)\operatorname{tr}(Y),
\end{equation}
where $\mu$ is a scalar parameter with $\mu > - 1/n$ to ensure positive-definiteness \cite{Pennec2006}. Therefore, the corresponding affine-invariant Riemannian metrics are generated by (\ref{isomorphism}) and given by
\begin{eqnarray}  \label{metrics}
\langle X, Y \rangle_{\Sigma} &= &\langle \Sigma^{-1/2}X\Sigma^{-1/2}, \Sigma^{-1/2}Y \Sigma^{-1/2}\rangle_{\mathfrak{m}} \nonumber  \\
& = & \operatorname{tr}(\Sigma^{-1}X\Sigma^{-1}Y)+\mu \operatorname{tr}(\Sigma^{-1}X)\operatorname{tr}(\Sigma^{-1}Y),
\end{eqnarray}
for $\Sigma \in S^+_n$ and $X,Y\in T_{\Sigma}S^+_n$. In the case $\mu = 0$, (\ref{metrics}) yields the most commonly used `natural' Riemannian metric on $S^+_n$, which corresponds to the Fisher information metric for the multivariate normal distribution \cite{Burbea1982,Skovgaard}, and has been widely used in applications such as tensor computing in medical imaging \cite{Barbaresco2008}.

\section{Affine-invariant orders}  \label{section 3}

\subsection{Affine-invariant cone fields}

A cone field $\mathcal{K}$ on $S^+_n$ smoothly assigns a cone $\mathcal{K}(\Sigma)\subset T_{\Sigma}S^+_n$ to each point $\Sigma \in S^+_n$. In this paper, we consider a cone to be a solid and pointed subset of a vector space that is closed under linear combinations with positive coefficients.
We say that $\mathcal{K}$ is affine-invariant or homogeneous with respect to the quotient geometry $S^+_n\cong GL(n)/O(n)$ if 
\begin{equation}
d\tau_A\big\vert_{\Sigma}\mathcal{K}(\Sigma)=\mathcal{K}(\tau_A(\Sigma)),
\end{equation}
for all $\Sigma\in S^+_n$ and $A\in GL(n)$. The procedure we will use for constructing affine-invariant cone fields on $S^+_n$ is similar to the approach taken for generating the affine-invariant Riemannian metrics in Section \ref{section 2}. We begin by defining a cone $\mathcal{K}(I)$ at $I$ that is $\operatorname{Ad}_{O(n)}$-invariant:
\begin{equation}   \label{Ad}
X \in \mathcal{K}(I) \Longleftrightarrow  \operatorname{Ad}_Q X = d\tau_Q\big\vert_I X = QXQ^T \in \mathcal{K}(I), \, \forall Q\in O(n).
\end{equation}
Using such a cone, we generate a cone field via 
\begin{equation}
\mathcal{K}(\Sigma)= d\tau_{\Sigma^{1/2}}\big\vert_I\mathcal{K}(I) =
\{X\in T_{\Sigma}S^+_n: \Sigma^{-1/2}X\Sigma^{-1/2}\in \mathcal{K}(I)\}.
\end{equation}
The $\operatorname{Ad}_{O(n)}$-invariance condition (\ref{Ad}) is satisfied if $\mathcal{K}(I)$ has a spectral characterization; that is, we can check to see if any given $X\in T_{I}S^+_n\cong \mathfrak{m}$ lies in $\mathcal{K}(I)$ using only properties of $X$ that are characterized by its spectrum. This observation leads to the following result.

\begin{proposition} \label{spectral}
A cone $\mathcal{K}(I)\in T_{I}S^+_n$ is $\operatorname{Ad}_{O(n)}$-invariant
if and only if there exists a cone $\mathcal{K}_{\Lambda}\subset\mathbb{R}^n$ that satisfies
\begin{equation} \label{permutation}
\boldsymbol{\lambda}\in\mathcal{K}_{\Lambda} \quad \Longleftrightarrow \quad P\boldsymbol{\lambda}\in\mathcal{K}_{\Lambda},
\end{equation}
for all permutation matrices $P\in \mathbb{R}^{n\times n}$, such that
$X\in \mathcal{K}(I)$ whenever $\boldsymbol{\lambda}_X\in\mathcal{K}_{\Lambda}$,
where $\boldsymbol{\lambda}_X=(\lambda_i(X))$ is a vector consisting of the  $n$ real eigenvalues of the symmetric matrix $X$.
\end{proposition}
For instance, $\operatorname{tr}(X)$ and $\operatorname{tr}(X^2)$ are both functions of $X$ that are spectrally characterized and indeed $\operatorname{Ad}_{O(n)}$-invariant. Quadratic  $\operatorname{Ad}_{O(n)}$-invariant cones are defined by inequalities on suitable linear combinations of  $(\operatorname{tr}(X))^2$ and $\operatorname{tr}(X^2)$.

\begin{proposition} 
For any choice of parameter $\mu \in (0,n)$, the set
\begin{equation}  \label{quad}
\mathcal{K}(I)=\{X\in T_IS^+_n:(\operatorname{tr}(X))^2-\mu \operatorname{tr}(X^2) \geq 0, \ \operatorname{tr}(X) \geq 0\},
\end{equation}
defines an $\operatorname{Ad}_{O(n)}$-invariant cone in $T_{I}S^+_n=\{X\in\mathbb{R}^{n\times n}: X^T=X\}$.
\end{proposition}

\begin{proof}
$\operatorname{Ad}_{O(n)}$-invariance is clear since $\operatorname{tr}(X^2)=\operatorname{tr}(QXQ^TQXQ^T)$ and $\operatorname{tr}(X)=\operatorname{tr}(QXQ^T)$ for all $Q\in O(n)$. To prove that (\ref{quad}) is a cone, first note that $0\in \mathcal{K}(I)$ and for $\lambda>0$, $X\in\mathcal{K}(I)$, we have $\lambda X \in \mathcal{K}(I)$ since $\operatorname{tr}(\lambda X)=\lambda \operatorname{tr}(X)\geq 0$ and
\begin{equation}
(\operatorname{tr}(\lambda X))^2-\mu \operatorname{tr}((\lambda X)^2)=\lambda^2[(\operatorname{tr}(X))^2-\mu \operatorname{tr}(X^2)] \geq 0.
\end{equation}
To show convexity, let $X_1, X_2\in\mathcal{K}(I)$. Now $\operatorname{tr}(X_1+X_2)=\operatorname{tr}(X_1)+\operatorname{tr}(X_2)\geq 0$, and
\begin{align}
(\operatorname{tr}(X_1+X_2)&)^2 -\mu \operatorname{tr}((X_1+X_2)^2)  =   [(\operatorname{tr}(X_1))^2-\mu \operatorname{tr}(X_1^2)] \nonumber \\ + \  \  & [(\operatorname{tr}(X_2))^2-\mu \operatorname{tr}(X_2^2)]   + 2 [\operatorname{tr}(X_1)\operatorname{tr}(X_2)-\mu \operatorname{tr}(X_1X_2)] \geq 0,
\end{align}
since 
$\operatorname{tr}(X_1X_2) \leq (\operatorname{tr}(X_1^2))^{\frac{1}{2}}(\operatorname{tr}(X_2^2))^{\frac{1}{2}} \leq \frac{1}{\sqrt{\mu}}\operatorname{tr}(X_1)\frac{1}{\sqrt{\mu}}\operatorname{tr}(X_2)$,  where the first inequality follows by Cauchy-Schwarz. Finally, we need to show that $\mathcal{K}(I)$ is pointed. If $X\in \mathcal{K}(I)$ and $-X \in \mathcal{K}(I)$, then $\operatorname{tr}(-X)=-\operatorname{tr}(X)=0$. Thus, $(\operatorname{tr}(X))^2-\mu \operatorname{tr}(X^2)=-\mu \operatorname{tr}(X^2)\geq 0$, which is possible if and only if all of the eigenvalues of $X$ are zero; i.e., if and only if X = 0.
\end{proof}

The parameter $\mu$ controls the opening angle of the cone. If $\mu = 0$, then (\ref{quad}) defines the half-space $\operatorname{tr}(X)\geq 0$. As $\mu$ increases, the opening angle of the cone becomes smaller and for $\mu = n$  (\ref{quad}) collapses to a ray.
For each $\mu\in (0,n)$, the cone $\mathcal{K}_{\Lambda}=\mathcal{K}_{\Lambda}^{\mu}\subset\mathbb{R}^n$ of Proposition \ref{spectral} is given by
\begin{equation} \label{spectral2}
\mathcal{K}_{\Lambda}^{\mu}=\Bigg\{\boldsymbol{\lambda}=(\lambda_i)\in\mathbb{R}^n:\left(\sum_{i=1}^n\lambda_i\right)^2-\mu\sum_{i=1}^n\lambda_i^2\geq 0,\,  \sum_{i=1}^n\lambda_i\geq 0\Bigg\},
\end{equation}
since $\tr(X)=\sum_{i=1}^n\lambda_i(X)$ and $\tr(X^2)=\sum_{i=1}^n\lambda_i^2(X)$. Indeed $\mathcal{K}_{\Lambda}^{\mu}$ is a quadratic cone 
\begin{equation} \label{spectral3}
\mathcal{K}_{\Lambda}^{\mu}=\{\boldsymbol{\lambda}\in\mathbb{R}^n:\boldsymbol{\lambda}^TQ_{\mu}\boldsymbol{\lambda}\geq 0,  \boldsymbol{1}^T\boldsymbol{\lambda}\,\geq 0\},
\end{equation}
where $\boldsymbol{1}=(1,\cdot\cdot\cdot, 1)^T\in\mathbb{R}^n$, and $Q_{\mu}$ is the $n\times n$ matrix with entries
$(Q_{\mu})_{ii}=
1-\mu$ and  
$(Q_{\mu})_{ij}=1$ for $i\neq j$. 

The dual cone $C^*$ of a subset $C$ of a vector space is a very important notion in convex analysis. For a vector space $\mathcal{V}$ endowed with an inner product $\langle \cdot,\cdot\rangle$, the dual cone can be defined as
$
C^*=\{y\in\mathcal{V}:\langle y,x\rangle\geq 0, \ \forall x\in C\}$. A cone is said to be self-dual if it coincides with its dual cone. It is well-known that the cone of positive semidefinite matrices is self-dual.
The following lemma will be used to characterize the form of the dual cone $(\mathcal{K}_{\Lambda}^{\mu})^*$ for each $\mu\in (0,n)$ with respect to the standard inner product on $\mathbb{R}^n$.

\begin{lemma}
The dual cone of the quadratic cone defined by (\ref{spectral3}) with respect to the standard inner product on $\mathbb{R}^n$ is given by
\begin{equation}
(\mathcal{K}_{\Lambda}^{\mu})^*=\{\boldsymbol{\lambda}\in\mathbb{R}^n:\boldsymbol{\lambda}^TQ_{\mu}^{-1}\boldsymbol{\lambda}\geq 0,  \boldsymbol{1}^T\boldsymbol{\lambda}\;\geq 0\}.
\end{equation}
\end{lemma}
The inverse matrix $Q_{\mu}^{-1}$ is given by
\begin{equation}
(Q_{\mu}^{-1})_{ij}=
\begin{cases}
\frac{\mu-(n-1)}{\mu(n-\mu)} \quad &i=j, \\
\frac{1}{\mu(n-\mu)} \quad &i\neq j.
\end{cases}
\end{equation}
Since $\mu(n-\mu)>0$  and $\mu-(n-1)=1-\mu^*$ where $\mu^*=n-\mu$, we find that
$\boldsymbol{\lambda}^TQ_{\mu}^{-1}\boldsymbol{\lambda}\geq 0$ if and only if $\boldsymbol{\lambda}^TQ_{\mu^*}\boldsymbol{\lambda}\geq 0$. That is,
\begin{equation} \label{dual}
(\mathcal{K}_{\Lambda}^{\mu})^*=\mathcal{K}_{\Lambda}^{n-\mu}.
\end{equation}
We notice of course from (\ref{dual}) that $\operatorname{Ad}_{O(n)}$-invariant cones are generally not self-dual. Indeed, for quadratic $\operatorname{Ad}_{O(n)}$-invariant cones, self-duality is only achieved for $\mu=n/2$.

Now for any fixed $\mu \in (0,n)$, we obtain a unique well-defined affine-invariant cone field given by
\begin{equation}  \label{quadField}
\mathcal{K}(\Sigma)=\{X\in T_{\Sigma}S^+_n:(\operatorname{tr}(\Sigma^{-1}X))^2-\mu \operatorname{tr}(\Sigma^{-1}X\Sigma^{-1}X) \geq 0, \ \operatorname{tr}(\Sigma^{-1}X) \geq 0\}.
\end{equation}
Note that for the value $\mu=0$, (\ref{quadField}) reduces to the affine-invariant half-space field
$\{X\in T_{\Sigma}S^+_n:\tr(\Sigma^{-1}X)\geq 0\}$. At the other extreme, for $\mu=n$, it is easy to show that the set at $I$ is given by the ray $\{X\in T_{I}S^+_n:X=\lambda I, \lambda\geq0\}$. By affine-invariance, (\ref{quadField}) reduces to $\{X\in T_{\Sigma}S^+_n:X=\lambda \Sigma, \lambda\geq0\}$ for $\mu=n$, which describes an affine-invariant field of rays in $S^+_n$. 

It should be noted that of course not all $\operatorname{Ad}_{O(n)}$-invariant cones at $I$ are quadratic.  Indeed, it is possible to construct polyhedral $\operatorname{Ad}_{O(n)}$-invariant cones that arise as the intersections of a collection of spectrally defined half-spaces in $T_{I}S^+_n$. The clearest example of such a construction is the cone of positive semidefinite matrices in $T_{I}S^+_n$, which of course itself has a spectral characterization $\mathcal{K}(I) = \{X\in T_I S^+_n: \lambda_i(X)\geq  0, \; i = 1, \dots, n\}$. 

\subsection{Affine-invariant pseudo-Riemannian structures on $S^+_n$} \label{pseudo}

At this point it is instructive to note the following systematic analysis of all affine-invariant pseudo-Riemannian structures on $S^+_n$ before continuing with our treatment of affine-invariant cone fields. This elegant characterization presents the affine-invariant Riemannian metrics of (\ref{metrics}) and the
quadratic affine-invariant cone fields of (\ref{quadField}) within a unified and rigorous mathematical framework. Recall that a pseudo-Riemannian metric is a generalization of a Riemannian metric in which the metric tensor need not be positive definite, but need only be a non-degenerate, smooth, symmetric  bilinear form. The signature of such a metric tensor is defined as the ordered pair consisting of the number of positive and negative eigenvalues of the real and symmetric matrix of the metric tensor with respect to a basis. Note that the signature of a metric tensor is independent of  the choice of basis by Sylvester's law of inertia. A metric tensor on a smooth manifold $\mathcal{M}$ is called Lorentzian if its signature is $(1,\dim\mathcal{M}-1)$.

The irreducible decomposition of $\mathfrak{m}$ under the $\operatorname{Ad}_{O(n)}$-action is given by $\mathfrak{m}=\mathbb{R}I\oplus\mathfrak{m}_0$, where $\mathfrak{m}_0:=\{X\in\mathfrak{m}:\tr X = 0\}$. According to this decomposition, we have $X=\frac{\tr X}{n}I\oplus\pi(X)$ for any $X\in\mathfrak{m}$, where $\pi(X):=X-\frac{\tr X}{n}I\in\mathfrak{m}_0$.
Denote by $\langle X, Y \rangle_{\mathrm{std}}$ the standard inner product $\tr(XY)$ on $\mathfrak{m}$, and let $\|X\|^2_{\mathrm{std}}:=\langle X,X\rangle_{\operatorname{std}}$ be the corresponding norm. Then we have
\begin{equation}
\tr (X^2) = \|X\|_{\mathrm{std}}^2 = \frac{(\tr X)^2}{n}+\|\pi(X)\|^2_{\mathrm{std}}.
\end{equation}
Now since $\mathfrak{m}_0$ is an irreducible $\operatorname{Ad}_{O(n)}$-module, any $\operatorname{Ad}_{O(n)}$-invariant quadratic form on $\mathfrak{m}_0$ is simply a scalar multiple of $\|\cdot\|^2_{\mathrm{std}}$ by Schur's lemma. Therefore, any $\operatorname{Ad}_{O(n)}$-invariant quadratic form on $\mathfrak{m}$ is of the form
\begin{equation}
Q_{\alpha\beta}(X):=\alpha \frac{(\tr X)^2}{n}+\beta \|\pi(X)\|^2_{\mathrm{std}},
\end{equation}
with $\alpha,\beta\in\mathbb{R}$. Clearly, $Q_{\alpha\beta}$ is positive definite if and only if $\alpha>0$ and $\beta>0$. Moreover, if $\alpha>0$ and $\beta<0$, then $Q_{\alpha\beta}$ is Lorentzian and the set $\{X\in\mathfrak{m}:Q_{\alpha\beta}(X)\geq 0,\, \tr X\geq0\}$ defines a pointed cone. Noting that
\begin{equation}
\tr(XY)+\mu\tr (X)\tr (Y) = \left(\mu +\frac{1}{n}\right)\tr (X) \tr (Y) + \langle \pi(X), \pi(Y) \rangle_{\mathrm{std}},
\end{equation}
for each $X, Y\in \mathfrak{m}$, we confirm that the metrics in (\ref{metrics}) are indeed positive definite if and only if $\mu > -1/n$. Similarly, we find that
\begin{equation}
(\tr X)^2-\mu\tr(X^2)=\frac{n-\mu}{n}(\tr X)^2 - \mu\|\pi(X)\|^2_{\mathrm{std}},
\end{equation}
which is Lorentzian if and only if $0<\mu<n$. Thus, we see that the affine-invariant pseudo-Riemannian structures on $S^+_n$ are essentially either Riemannian or Lorentzian, and the quadratic cone fields in (\ref{quadField}) are precisely the cone fields defined by the affine-invariant Lorentzian metrics. 

\subsection{Affine-invariant partial orders on $S^+_n$}\label{partial orders}

A smooth cone field $\mathcal{K}$ on a manifold $\mathcal{M}$ gives rise to a \emph{conal order} $\prec_{\mathcal{K}}$ on $\mathcal{M}$, defined by $x\prec_{\mathcal{K}} y$ if there exists a (piecewise) smooth curve $\gamma:[0,1]\rightarrow \mathcal{M}$ with $\gamma(0)=x$, $\gamma(1)=y$ and $\gamma'(t)\in\mathcal{K}(\gamma(t))$ whenever the derivative exists. The closure $\leq_{\mathcal{K}}$ of this order is again an order and satisfies $x\leq_{\mathcal{K}} y$ if and only if $y\in\overline{\{z:x\prec_{\mathcal{K}} z\}}$. We say that $\mathcal{M}$ is \emph{globally orderable} if $\leq_{\mathcal{K}}$ is a partial order. Here we will prove that the conal orders induced by affine-invariant cone fields on $S^+_n$ define partial orders. That is, we will show that the conal orders satisfy the antisymmetry property that $\Sigma_1\leq_{\mathcal{K}} \Sigma_2$ and $\Sigma_2\leq_{\mathcal{K}} \Sigma_1$ together imply $\Sigma_1=\Sigma_2$, for any affine-invariant cone field $\mathcal{K}$ on $S^+_n$. In other words, we will prove that there do not exist any non-trivial closed conal curves in $S^+_n$. In the following, we will make use of the preimage theorem \cite{Banyaga2013} given below. Recall that given a smooth map $F:\mathcal{M}\rightarrow\mathcal{N}$ between manifolds, we say that a point $y\in\mathcal{N}$ is a \emph{regular value} of $F$ if for all $x\in F^{-1}(y)$ the map $dF\vert_x:T_x\mathcal{M}\rightarrow T_y\mathcal{N}$ is surjective.

\begin{theorem} [The preimage theorem] \label{preimage}
Let $F:\mathcal{M}\rightarrow\mathcal{N}$ be a smooth map of manifolds, with $\dim\mathcal M=m$ and $\dim\mathcal{N}=n$. If $x\in\mathcal{N}$ is a regular value of $F$, then $F^{-1}(c)$ is a submanifold of $\mathcal{M}$ of dimension $m-n$. Moreover, the tangent space of $F^{-1}(c)$ at $x$ is equal to $\ker(dF\vert_x)$. 
\end{theorem}

Now define $F:S^+_n\rightarrow\mathbb{R}$ by $F(\Sigma)=\det\Sigma$. By Jacobi's formula, the differential of the determinant takes the form
$
d(\det)\vert_{\Sigma}X=\tr\left(\operatorname{adj}(\Sigma)X\right),
$
where $\operatorname{adj}(\Sigma)$ denotes the adjugate of $\Sigma$. That is,
\begin{equation} \label{Jacobi}
dF\vert_{\Sigma}X=(\det\Sigma)\tr\left(\Sigma^{-1}X\right),
\end{equation}
for all $X\in T_{\Sigma}S^+_n$. Note that for $c>0$ and any $\Sigma\in F^{-1}(c)$, we have
$dF\vert_{\Sigma}I=c\tr\left(\Sigma^{-1}\right)>0$, which clearly shows that any $c>0$ is a regular value of $F$. Hence, $F^{-1}(c)$ is a submanifold of codimension $1$ for any choice of $c>0$. Furthermore, as $\operatorname{im}(F) = \mathbb{R}^+=\{c\in\mathbb{R}:c>0\}$, the collection of submanifolds $\{F^{-1}(c)\}_{c>0}$ forms a foliation of $S^+_n$. Since $\det\Sigma>0$ for any $\Sigma\in S^+_n$, (\ref{Jacobi}) implies that $\operatorname{ker}(dF\vert_{\Sigma})=\{X\in T_{\Sigma}S^+_n:\tr(\Sigma^{-1}X)=0\}$. Thus, the tangent spaces to the submanifolds $\{F^{-1}(c)\}_{c>0}$ are described by 
the affine-invariant distribution  $\mathcal{D}_{\Sigma}$ of rank $\dim S^+_n-1 = n(n+1)/2-1$ on $S^+_n$ defined by $\mathcal{D}_{\Sigma}:=\{X\in T_{\Sigma}S^+_n:\tr(\Sigma^{-1}X)=0\}$.

\begin{proposition} \label{det}
If $\gamma:[0,1]\rightarrow S^+_n$ is a non-trivial conal curve with respect to a quadratic affine-invariant cone field $\mathcal{K}$ (\ref{quadField}), then
\begin{equation}
t_2>t_1 \implies \det(\gamma(t_2))>\det(\gamma(t_1)),
\end{equation}
for $t_1,t_2\in [0,1]$.
\end{proposition}

\begin{proof}
First note that $X\in\mathcal{K}(\Sigma)\setminus\{0\}$ implies that $\tr(\Sigma^{-1}X)>0$.
This follows by noting that if $\tr(\Sigma^{-1}X)=0$, then $\tr(\Sigma^{-1}X\Sigma^{-1}X)=\tr[(\Sigma^{-1/2}X\Sigma^{-1/2})^2]\leq0$, which is a contradiction. 
For simplicity, we assume that $\gamma$ is a non-trivial smooth conal curve. The proof for a piecewise smooth curve is similar. We then have $\tr(\gamma(t)^{-1}\gamma'(t))>0$, which implies that
\begin{equation} \label{der det}
\frac{d}{dt}\det\gamma(t)=(\det\gamma(t))\tr\left(\gamma(t)^{-1}\gamma'(t)\right)>0.
\end{equation}
\end{proof}

Proposition \ref{det} clearly implies that $S^+_n$ equipped with any of the cone fields described by  (\ref{quadField}) does not admit any non-trivial closed conal curves. Indeed, this result holds for all affine-invariant cone fields, not just quadratic ones. To see this, note that the permutation symmetry (\ref{permutation}) of Proposition \ref{spectral}, implies that $\tr(\Sigma^{-1}X)\neq 0$ whenever $X\in\mathcal{K}(\Sigma)\setminus\{0\}$.
It thus follows by (\ref{der det}) that $\det\circ\gamma:[0,1]\rightarrow\mathbb{R}^+$ is a strictly monotone function for any non-trivial conal curve $\gamma$, which rules out the existence of closed conal curves. We thus arrive at the following theorem.

\begin{theorem}
All affine-invariant conal orders on $S^+_n$ are partial orders.
\end{theorem}

At this point it is worth noting a few interesting features of the collection of submanifolds $\{F^{-1}(c)\}_{c>0}$ of $S^+_n$. First note that if $\gamma$ is an \emph{inextensible} conal curve, then by (\ref{der det}) it must intersect each of the submanifolds $F^{-1}(c)$ exactly once. That is, for each $c>0$, $F^{-1}(c)$ defines a \emph{Cauchy surface} for the causal structure induced by any affine-invariant cone field. We also note the following results which connect these submanifolds to geodesics on $S^+_n$ with respect to the standard affine-invariant Riemannian metric $ds^2=\tr[(\Sigma^{-1}d\Sigma)^2]$ on $S^+_n$.

\begin{proposition}
Endow $S^+_n$ with the Riemannian structure defined by the standard Riemannian metric $ds^2=\tr[(\Sigma^{-1}d\Sigma)^2]$. We have the following results.
\begin{enumerate}[i)]
  \item If $\Sigma_1,\Sigma_2\in S^+_n$ satisfy $\det \Sigma_1=\det \Sigma_2=c$, then the geodesic from $\Sigma_1$ to $\Sigma_2$ lies in $F^{-1}(c)$.
  \item If $X\in T_{\Sigma}S^+_n$ satisfies $\tr(\Sigma^{-1}X)=0$, then the geodesic through $\Sigma$ in the direction of $X$ stays on the submanifold $F^{-1}(\det\Sigma)$.
\end{enumerate}
\end{proposition}

\begin{proof}
$i)$ Let $\Sigma_1,\Sigma_2\in S^+_n$ satisfy $\det\Sigma_1=\det\Sigma_2$. The geodesic $\gamma$ from $\Sigma_1$ to $\Sigma_2$ is given by
\begin{equation} \label{geodesic from point to point}
\gamma(t)=\Sigma_1^{1/2}\exp\left(t\log\left(\Sigma_1^{-1/2}\Sigma_2\Sigma_1^{-1/2}\right)\right)\Sigma_1^{1/2}.
\end{equation}
Thus, $\det(\gamma(t))=(\det\Sigma_1)\det(\exp(t\log(\Sigma_1^{-1/2}\Sigma_2\Sigma_1^{-1/2}))$. Using the matrix identity $\log(\det A)=\tr(\log A)$, we find that
\begin{align}
\log\left[\det\left(\exp\left(t\log\left(\Sigma_1^{-\frac{1}{2}}\Sigma_2\Sigma_1^{-\frac{1}{2}}\right)\right)\right)\right]&=
\tr\left[\log\left(\exp\left(t\log\left(\Sigma_1^{-\frac{1}{2}}\Sigma_2\Sigma_1^{-\frac{1}{2}}\right)\right)\right)\right] \nonumber \\
&= t \tr\left(\log\left(\Sigma_1^{-\frac{1}{2}}\Sigma_2\Sigma_1^{-\frac{1}{2}}\right)\right) \\
&= t \log\left(\det\Sigma_2/\det\Sigma_1\right) = 0.
\end{align}
Therefore, $\det(\exp(t\log(\Sigma_1^{-1/2}\Sigma_2\Sigma_1^{-1/2}))=1$, which implies that $\det(\gamma(t))=\det\Sigma_1$ for all $t\in\mathbb{R}$.

$ii)$ The geodesic $\gamma$ from $\Sigma$ in the direction of $X\in T_{\Sigma}S^+_n$ takes the form
$\gamma(t)=\Sigma^{1/2}\exp(t\Sigma^{-1/2}X\Sigma^{-1/2})\Sigma^{1/2}$. If $\tr(\Sigma^{-1}X)=0$, then
\begin{equation}
\log(\det(\exp(t\Sigma^{-1/2}X\Sigma^{-1/2})))=\tr(t\Sigma^{-1/2}X\Sigma^{-1/2})=t\tr(\Sigma^{-1}X)=0,
\end{equation}
which implies that $\det(\gamma(t))=(\det\Sigma)\det(\exp(t\Sigma^{-1/2}X\Sigma^{-1/2}))=\det\Sigma$ for all $t\in\mathbb{R}$.
\end{proof}

\subsection{Causal semigroups} \label{causal semi}

Define a \emph{wedge} to be a closed and convex subset of a vector space that is also invariant with respect to scaling by positive numbers. Notice in particular that a wedge need not be pointed. Let $\mathcal{M}=G/H$ be a homogeneous space, $G$ a Lie group with group identity element $e$ and Lie algebra $\mathfrak{g}$, $H$ a closed subgroup with Lie algebra $\mathfrak{h}$, and $\pi:G\rightarrow \mathcal{M}$ the associated projection map. Assume that the Lie algebra $\mathfrak{g}$ contains a wedge $W$ such that $(i)$ $W\cap-W=\mathfrak{h}$ and $(ii)$ $\operatorname{Ad}(h)W=W$ for all $h\in H$. A wedge $W$ is said to be a \emph{Lie wedge} if $e^{\operatorname{ad}h}W=W$ for all $h\in W\cap-W$. Denoting the left action of $G$ on $\mathcal{M}$ by  $\tau_g:\mathcal{M}\rightarrow\mathcal{M}$, we have $\pi\circ\lambda_g=\tau_g\circ \pi$, where $\lambda_g$ is the left multiplication with $g$ on $G$. Conditions $(i)$ and $(ii)$ ensure that $d\pi\vert_g\circ d\lambda_g\vert_eW$ only depends on $\pi(g)$, so that 
\begin{equation}
\mathcal{K}(\pi(g))=\left(d\pi\vert_g\circ d\lambda_g\vert_e\right)W,
\end{equation}
yields a well-defined field of pointed cones on $\mathcal{M}$ that is
invariant under the action of $G$ on $\mathcal{M}$: $d\tau_g\vert_x\mathcal{K}(x)=\mathcal{K}(\tau_g(x))$. These results can be found in \cite{Hilgert1989}.
The set $S=\{g\in G: o \leq_{\mathcal{K}} \tau_g(o)\}$, where $o=\pi(e)$, is a closed semigroup of $G$ referred to as the causal semigroup of $(\mathcal{M},G,\mathcal{K})$.
 The following theorem is derived from \cite{Neeb1991}.

\begin{theorem} \label{Causal semigroup thm}
Let $S=\overline{\langle\exp W\rangle H}\subseteq G$, then
$S=\pi^{-1}\left(\{x\in\mathcal{M}:o\leq_{\mathcal{K}} x\}\right)$ and $\mathcal{M}$ is globally orderable with respect to $\mathcal{K}$ if and only if $W=\boldsymbol{L}(S)$, where
\begin{equation}
\boldsymbol{L}(S)=\{Z\in\mathfrak{g}:\exp(\mathbb{R}^+Z)\subseteq S\}.
\end{equation}
\end{theorem}

The affine-invariant cone fields on $S^+_n=GL(n)/O(n)$ can be viewed as projections of invariant wedge fields on the Lie group $GL(n)$ in the sense of the above results. Since we have the reductive decomposition $\mathfrak{gl}(n)=\mathfrak{o}(n)\oplus\mathfrak{m}$, it is easy to construct the corresponding wedge field $W$ that satisfies conditions $(i)$ and $(ii)$ for a given affine-invariant cone field $\mathcal{K}$. We will now use this structure and Theorem \ref{Causal semigroup thm} to prove the following important result.

\begin{theorem} \label{geodesic conal curves thm}
Let $S^+_n$ be equipped with an affine-invariant cone field $\mathcal{K}$ and the standard affine-invariant Riemannian metric $ds^2=\tr[(\Sigma^{-1}d\Sigma)^2]$. For any pair of matrices $\Sigma_1,\Sigma_2\in S^+_n$, we have $\Sigma_1\leq_{\mathcal{K}}\Sigma_2$ if and only if the geodesic from $\Sigma_1$ to $\Sigma_2$ is a conal curve.
\end{theorem}

\begin{proof}
Note that the expression of the geodesic from $\Sigma_1$ to $\Sigma_2$ given in (\ref{geodesic from point to point}) implies that this theorem is equivalent to
\begin{equation}  \label{inv order condition}
\Sigma_1\leq_{\mathcal{K}}\Sigma_2 \quad \Longleftrightarrow  \quad \log\left(\Sigma_1^{-1/2}\Sigma_2\Sigma_1^{-1/2}\right)\in\mathcal{K}(I).
\end{equation}
As $\mathcal{K}$ is affine-invariant, $\Sigma_1\leq_{\mathcal{K}}\Sigma_2$ is equivalent to $I\leq \Sigma_1^{-1/2}\Sigma_2\Sigma_1^{-1/2}$. Thus, it is sufficient to prove that
\begin{equation}
I\leq_{\mathcal{K}}\Sigma \quad \Longleftrightarrow \quad \log\left(\Sigma\right)\in\mathcal{K}(I),
\end{equation}
for any $\Sigma\in S^+_n$.
We define a wedge $W$ in $\mathfrak{gl}(n)$ by
\begin{equation} \label{W def}
W:=\{X+Y: X\in\mathcal{K}(I), Y\in\mathfrak{o}(n)\}\subset\mathfrak{gl}(n)=\mathfrak{m}\oplus\mathfrak{o}(n),
\end{equation}
where $\mathcal{K}(I)$ is viewed as a subset of $\mathfrak{m}\cong T_I S^+_n$. Note that (\ref{W def}) ensures that $W$ satisfies the properties required of it in Theorem \ref{Causal semigroup thm}. If $I\leq_{\mathcal{K}}\Sigma$, it follows from Theorem \ref{Causal semigroup thm}  that there exists $A\in W$ such that 
\begin{equation} \label{pi exp}
\Sigma=\pi(\exp A)=\tau_{\exp A}(I)=(\exp A)(\exp A)^T.
\end{equation}
By the polar decomposition theorem of \cite{Lawson1991}, any element $g=\exp A$ of the semigroup $S=\overline{\langle\exp W\rangle O(n)}\subset GL(n)$ admits a unique decomposition as $g=(\exp X)Q$ with $X\in W\cap\mathfrak{m}$ and $Q\in O(n)$. Thus, we have
\begin{equation}
\Sigma=\tau_g(I)=\tau_{\exp X}(I)=\exp 2X,
\end{equation}
so that $\log\Sigma=2X\in\mathcal{K}(I)$.
\end{proof}

\begin{remark}
Let $\mathcal{K}$ be a quadratic affine-invariant cone field described by (\ref{quadField}). Given a pair $\Sigma_1, \Sigma_2\in S^+_n$, we have by Theorem \ref{geodesic conal curves thm} that $\Sigma_1\leq_{\mathcal{K}} \Sigma_2$ if and only if $\log\left(\Sigma_1^{-1/2}\Sigma_2\Sigma_1^{-1/2}\right)\in\mathcal{K}(I)$, which is equivalent to
\begin{equation} \label{conic 4}
\begin{cases}
\tr\left(\log(\Sigma_1^{-1/2}\Sigma_2\Sigma_1^{-1/2})\right)\geq 0, \\
\left(\tr(\log(\Sigma_1^{-1/2}\Sigma_2\Sigma_1^{-1/2}))\right)^2-\mu\tr\left[(\log(\Sigma_1^{-1/2}\Sigma_2\Sigma_1^{-1/2}))^2\right]\geq 0.
\end{cases}
\end{equation}
Since $\Sigma_1^{-1/2}\Sigma_2\Sigma_1^{-1/2}$ and $\Sigma_2\Sigma_1^{-1}$ have the same spectrum, (\ref{conic 4}) can be written as
\begin{equation} \label{conic 5}
\begin{cases}
\tr\left(\log(\Sigma_2\Sigma_1^{-1})\right)\geq 0, \\ \left(\tr(\log(\Sigma_2\Sigma_1^{-1}))\right)^2-\mu\tr\left[(\log(\Sigma_2\Sigma_1^{-1}))^2\right]\geq 0,
\end{cases}
\end{equation}
which has the virtue of not involving square roots of $\Sigma_1$ and $\Sigma_2$. Equation (\ref{conic 5}) in turn is equivalent to
\begin{equation} \label{conic 6}
\begin{cases}
\sum_i\log\lambda_i\geq 0, \\ \left(\sum_i\log\lambda_i\right)^2-\mu\sum_i(\log\lambda_i)^2\geq 0,
\end{cases}
\end{equation}
where $\lambda_i=\lambda_i(\Sigma_2\Sigma_1^{-1})$ $(i=1,...,n)$ denote the $n$ real and positive eigenvalues of $\Sigma_2\Sigma_1^{-1}$. We have thus used invariance to reduce the question of whether a pair of positive definite matrices $\Sigma_1$ and $\Sigma_2$ are ordered with respect to any of the quadratic affine-invariant cone fields
to a pair of inequalities involving the spectrum of $\Sigma_2\Sigma_1^{-1}$.
\end{remark}

\subsection{Visualization of affine-invariant cone fields on $S^+_2$} \label{visualization}

It is well-known that the set of positive semidefinite matrices of dimension $n$ forms a cone in the space of symmetric $n\times n$ matrices. Moreover, $S^+_n$ forms the interior of this cone.
A concrete visualization of this identification can be made in the $n=2$ case, as shown in Figure \ref{fig:1} $(a)$. The set $S^+_2$
can be identified with the interior of the set
$K=\{(x,y,z)\in \mathbb{R}^3: z^2-x^2-y^2 \geq 0, \ z \geq 0\}$,
through the bijection $\phi:S^+_2\to \operatorname{int} K$ given by
\begin{equation}  \label{bijection}
\phi: \begin{pmatrix}
a \ & \ b \\
b \ & \ c  
\end{pmatrix} \mapsto (x,y,z)=\left(\sqrt{2}b,\frac{1}{\sqrt{2}}(a-c),\frac{1}{\sqrt{2}}(a+c)\right).
\end{equation}
Inverting $\phi$, we find that $a = (z+y)/\sqrt{2}$, $b=x/\sqrt{2}$, $c=(z-y)/\sqrt{2}$. Note that the point $(x,y,z)=(0,0,\sqrt{2})$ corresponds to the identity matrix $I\in S^+_2$. We seek to arrive at a visual representation of the affine-invariant cone fields generated from the $\operatorname{Ad}_{O(n)}$-invariant cones (\ref{quad}) for different choices of the parameter $\mu$. The defining inequalities $\operatorname{tr}(X)\geq 0$ and $(\operatorname{tr}(X))^2 - \mu \operatorname{tr}(X^2) \geq 0$ in $T_IS^+_2$ take the forms
\begin{equation}  \label{coord}
\delta z \geq 0, \quad \textrm{and} \quad \left(\frac{2}{\mu}-1\right)\delta z^2 - \delta x^2-\delta y^2 \geq 0,
\end{equation}
respectively, where $(\delta x, \delta y, \delta z)\in T_{(0,0,\sqrt{2})} K\cong T_{I}S^+_2$. The corresponding spectral cone $\mathcal{K}_{\Lambda}^{\mu}\subset \mathbb{R}^2$ is given by
\begin{equation}
\lambda_1+\lambda_2\geq 0, \quad \textrm{and} \quad (\lambda_1+\lambda_2)^2-\mu(\lambda_1^2+\lambda_2^2)\geq 0.
\end{equation}
See Figure \ref{fig:1} $(b)$ for an illustration of such a cone for a choice of $\mu\in(0,1)$.

\begin{figure} 
\centering
\includegraphics[width=0.9\linewidth]{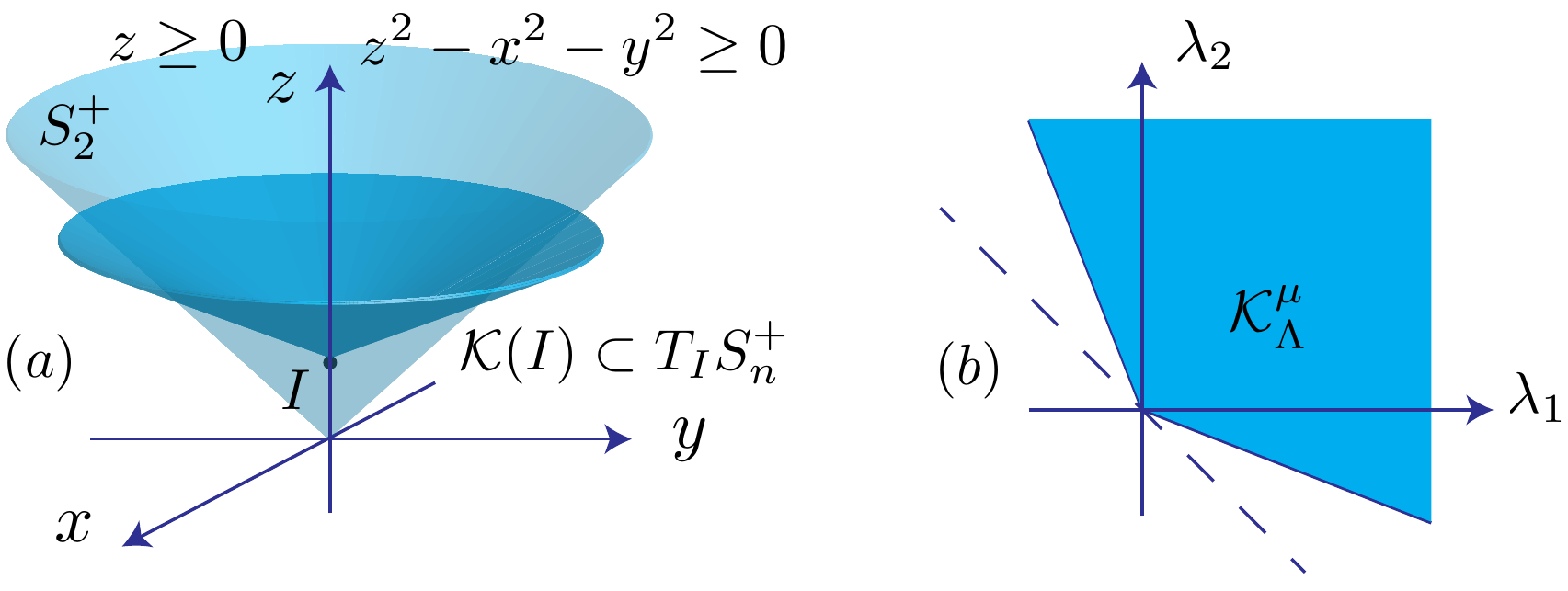}
  \caption{$(a)$ Identification of $S^+_2$ with the interior of the closed, convex, pointed cone $K=\{(x,y,z)\in\mathbb{R}^3:z^2-x^2-y^2  \geq 0$, $z\geq 0$\} in $\mathbb{R}^3$. The $\operatorname{Ad}_{O(n)}$-invariant cone $\mathcal{K}(I)\subset T_{I}S^+_n$ at identity is also shown for a choice of $\mu\in(0,1)$. $(b)$ The corresponding spectral cone $\mathcal{K}^{\mu}_{\Lambda}\subset\mathbb{R}^2$ which characterizes the cone $\mathcal{K}(I)\subset T_{I}S^+_n$.
  }
  \label{fig:1}
\end{figure}

Clearly the translation invariant cone fields generated from this cone are given by the same equations as in (\ref{coord}) for $(\delta x, \delta y, \delta z)\in T_{(x,y,z)} K\cong T_{\Sigma}S^+_2$, where $\phi(\Sigma)=(x,y,z)$.
To obtain the affine-invariant cone fields, note that at $\Sigma = \phi^{-1}(x,y,z)\in S^+_2$, the inequality $\operatorname{tr}(\Sigma^{-1}X)\geq 0$ takes the form
\begin{align}
 \operatorname{tr}\left[\begin{pmatrix}
 c & -b \\
 -b & a
 \end{pmatrix}
 \begin{pmatrix}
 \delta a \ & \ \delta b \\
 \delta b \ & \ \delta c
 \end{pmatrix}
 \right]= c\ \delta a - 2b \ \delta b +a \  \delta c \geq 0 \\
\Longleftrightarrow \ z\ \delta z - x \ \delta x - y \ \delta y \geq 0.  \label{halfspace}
\end{align}
Similarly, the inequality $(\operatorname{tr}(\Sigma^{-1}X))^2-\mu \operatorname{tr}(\Sigma^{-1}X\Sigma^{-1}X)\geq 0$ is equivalent to
\begin{align}  \label{coord2}
 2(x\ \delta x+y\ \delta y - z\ \delta z)^2 - &\mu\left[(z^2+x^2-y^2)\delta x^2+(z^2-x^2-y^2)\delta y^2 \right.
 \nonumber \\
 + \ (x^2+y^2+z^2)\delta z^2 &\left.+ \ 4xy\ \delta x \delta y - 4xz\ \delta x \delta z - 4yz\ \delta y \delta z\right] \geq 0,
\end{align}
where $(\delta x, \delta y, \delta z)\in T_{(x,y,z)} K\cong T_{\Sigma}S^+_2$. In the case $\mu=1$, this reduces to $(\frac{2}{\mu}-1)\delta z^2-\delta x^2-\delta y^2 \geq 0$. Thus, for $\mu =1$ the quadratic cone field generated by affine-invariance coincides with the corresponding translation-invariant cone field. Generally, however, affine-invariant and translation-invariant cone fields do not agree, as depicted in Figure \ref{fig:2}. Each of the distinct cone fields in Figure \ref{fig:2} induces a distinct partial order on $S^+_n$.

\begin{figure}
\centering
\includegraphics[width=0.95\linewidth]{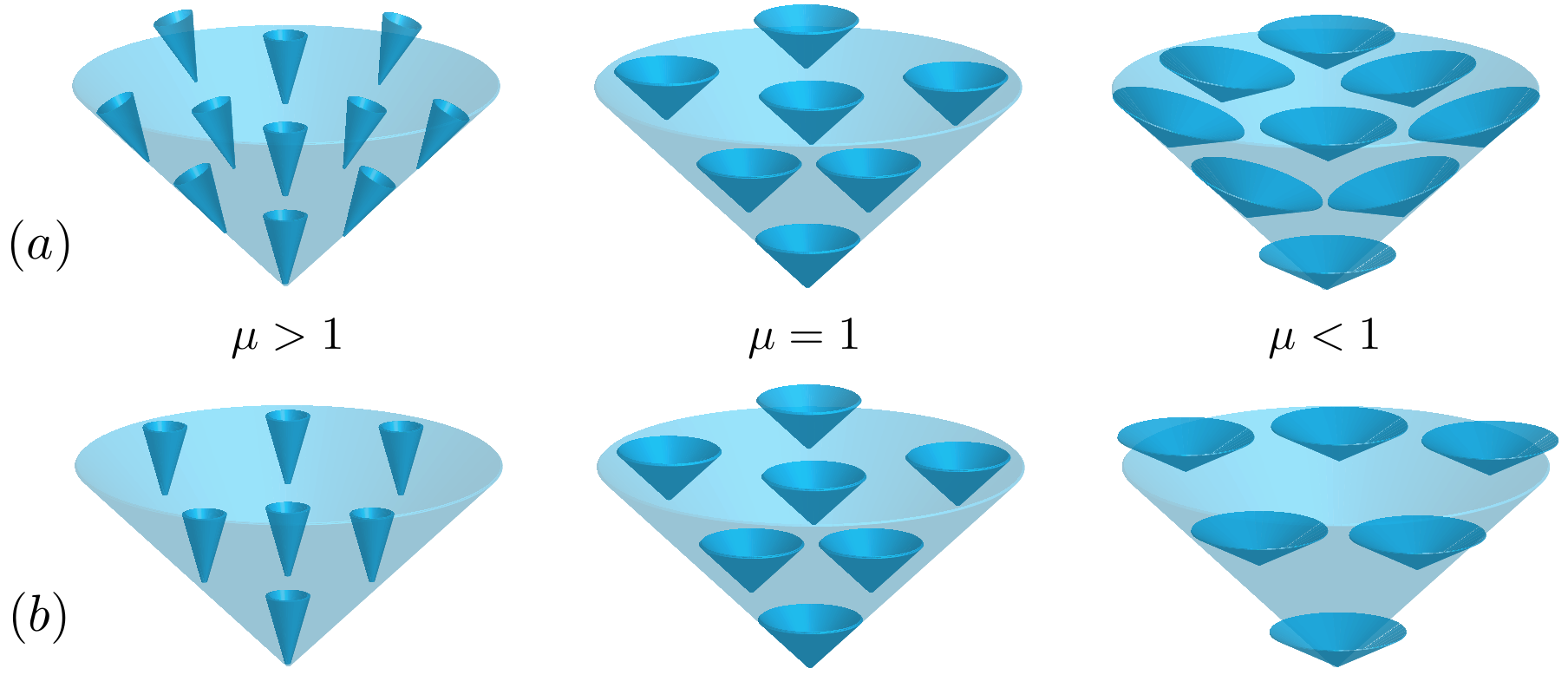}
  \caption{Cone fields on $S^+_2$: $(a)$ Quadratic affine-invariant cone fields for different choices of the parameter $\mu \in (0,2)$. $(b)$ The corresponding translation-invariant cone fields.
  }
  \label{fig:2}
\end{figure}

\subsection{The L\"owner order}

The L\"owner order is the partial order $\geq_L$ on $S^+_n$ defined by
\begin{equation}  \label{Lowner}
A \geq_L B \quad \Longleftrightarrow \quad A-B \geq_L O,
\end{equation}
where the inequality on the right denotes that $A-B$ is positive semidefinite \cite{Bhatia}. The definition in (\ref{Lowner}) is based on translations and the `flat' geometry of $S^+_n$. It is clear that the L\"owner order is translation invariant in the sense that $A\geq_L B$ implies that $A+C \geq_L B +C$ for all $A,B,C\in S^+_n$. From the perspective of conal orders, the L\"owner order is the partial order induced by the cone field generated by translations of the cone of positive semidefinite matrices at $T_IS^+_n$. 

In the previous section, we gave an explicit construction showing that the cone field generated through translations of the cone of positive semidefinite matrices at $T_IS^+_n$ coincides with the cone field generated through affine-invariance in the $n=2$ case. We will now show that this is a general result which holds for all $n$. First note that the cone at $T_IS^+_n$ can be expressed as
\begin{equation}
\mathcal{K}(I)=\{X\in T_IS^+_n: u^TXu \geq 0 \  \forall u\in \mathbb{R}^n,\  u^TXu = 0 \Rightarrow u =0\},
\end{equation}
and the resulting translation-invariant cone field is simply given by
\begin{equation}
\mathcal{K}_T(\Sigma)=\{X\in T_{\Sigma}S^+_n: u^TXu \geq 0 \  \forall u\in \mathbb{R}^n,\  u^TXu = 0 \Rightarrow u =0\}.
\end{equation}
The corresponding affine-invariant cone field is given by
\begin{align}  \label{A}
\mathcal{K}_A(\Sigma)=\{X\in T_{\Sigma}S^+_n: u^T\Sigma^{-1/2}X\Sigma^{-1/2}u \geq 0 \  & \forall u\in \mathbb{R}^n, \nonumber \\
\  u^T\Sigma^{-1/2}&X\Sigma^{-1/2}u = 0 \Rightarrow u =0\},
\end{align}
which is seen to be equal to $\mathcal{K}_T$ by introducing the invertible transformation $\bar{u}=\Sigma^{-1/2}u$ in (\ref{A}). Thus we see that the L\"owner order enjoys the special status of being both affine-invariant and translation-invariant, even though its classical definition is based on the `flat' or translational geometry on $S^+_n$.

\section{Monotone functions on $S^+_n$}  \label{section 4}

\subsection{Differential positivity}

Let $f$ be a map of $S^+_n$ into itself. We say that $f$ is \emph{monotone} with respect to a partial order $\geq$ on $S^+_n$ if $f(\Sigma_1)\geq f(\Sigma_2)$ whenever $\Sigma_1\geq \Sigma_2$.  Such functions were introduced by L\"owner in his seminal paper \cite{Lowner1934} on operator monotone functions. Since then operator monotone functions have been studied extensively and found applications to many fields including electrical engineering \cite{Anderson1976}, network theory, and quantum information theory \cite{Bhatia2013, Nielsen2002}. Monotonicity of mappings and dynamical systems with respect to partial orders induced by cone fields have a local geometric characterization in the form of differential positivity \cite{Forni2015}.  
A smooth map $f:S^+_n\to S^+_n$ is said to be differentially positive with respect to a cone field $\mathcal{K}$ on $S^+_n$ if
$df\vert_{\Sigma}(\delta\Sigma)\in\mathcal{K}(f(\Sigma))$ whenever $\delta\Sigma\in\mathcal{K}(\Sigma)$,
where $df\vert_{\Sigma}:T_{\Sigma}S^+_n \to T_{f(\Sigma)}S^+_n$ denotes the differential of $f$ at $\Sigma$. Assuming that $\geq_{\mathcal{K}}$ is a partial order induced by $\mathcal{K}$, then $f$ is monotone with respect to $\geq_{\mathcal{K}}$ if and only if it is differentially positive with respect to $\mathcal{K}$. To see this, recall that $\Sigma_2\geq_{\mathcal{K}} \Sigma_1$ means that there exists some conal curve $\gamma:[0,1]\to S^+_n$ such that $\gamma(0)=\Sigma_ 1$, $\gamma(1)=\Sigma_2$ and $\gamma'(t)\in \mathcal{K}(\gamma(t))$ for all $t\in (0,1)$. Now $f\circ\gamma: [0,1]\to S^+_n$ is a curve in $S^+_n$ with $(f\circ\gamma) (0)=f(\Sigma_1)$, $(f\circ\gamma)(1)=f(\Sigma_2)$, and
\begin{equation}
(f\circ\gamma)'(t)=df\vert_{\gamma(t)}\gamma'(t).
\end{equation}
Hence, $f\circ\gamma$ is a conal curve joining $f(\Sigma_1)$ to $f(\Sigma_2)$ if and only if 
$df\vert_{\gamma(t)}\mathcal{K}(\gamma(t))\subseteq\mathcal{K}(f(\gamma(t))$. 

\subsection{The Generalized L\"owner-Heinz Theorem}

One of the most fundamental results in operator theory is the  L\"owner-Heinz theorem  \cite{Lowner1934,Heinz1951} stated below.

\begin{theorem} [L\"owner-Heinz] \label{LownerHeinz}
If $\Sigma_1\geq_L \Sigma_2$ in $S^+_n$ and $r\in[0,1]$, then 
\begin{equation}
\Sigma_1^r\geq_L\Sigma_2^r.
\end{equation}
\end{theorem}
Furthermore, if $n\geq 2$ and $r>1$, then $\Sigma_1\geq_L \Sigma_2 \not \Rightarrow \Sigma_1^r\geq_L \Sigma_2^r$.  \\

There are several different proofs of the L\"owner-Heinz theorem. See \cite{Bhatia, Pedersen1972, Lowner1934, Heinz1951}, for instance. Most of these proofs are based on analytic methods, such as integral representations from complex analysis. Instead we employ a geometric approach to study monotonicity based on a differential analysis of the system. One of the advantages of such an approach is that it is immediately applicable to all of the conal orders considered in this paper, while providing geometric insight into the behavior of the map under consideration. By using invariant differential positivity with respect to the family of affine-invariant cone fields in (\ref{quadField}), we arrive at the following extension to the L\"owner-Heinz theorem.

\begin{theorem} [Generalized L\"owner-Heinz] \label{GenLownerHeinz}
For any of the affine-invariant partial orders induced by the quadratic cone fields (\ref{quadField}) parametrized by $\mu$, the map  $f_r(\Sigma)=\Sigma^r$ is monotone on $S^+_n$ for any $r\in [0,1]$.
\end{theorem}

This result suggests that the monotonicity of the map $f_r: \Sigma \mapsto \Sigma^r$ for $r\in(0,1)$ is intimately connected to the affine-invariant geometry of $S^+_n$ and not its translational geometry. The structure of the proof of Theorem \ref{GenLownerHeinz} is as follows.
We first prove that the map $f_{1/p}:\Sigma\mapsto\Sigma^{1/p}$ is monotone for any $p\in\mathbb{N}$. We then extend this result to maps $f_{q/p}:\Sigma\mapsto\Sigma^{q/p}$ for rational numbers $q/p\in\mathbb{Q}\cap(0,1)$, before arriving at the full result via a density argument.
We prove monotonicty by establishing differential positivity in each case. To prove the monotonicity of  $f_{1/p}:\Sigma\mapsto\Sigma^{1/p}$, $p\in\mathbb{N}$, we only need the following lemma \cite{Yang2002}.

\begin{lemma} \label{inequalitylemma}
If $A$ and $B$ are Hermitian $n\times n$ matrices, then
\begin{equation}
\operatorname{tr}[(AB)^{2m}]\leq \tr[A^{2m}B^{2m}], \quad m\in\mathbb{N}.
\end{equation}
\end{lemma}

The proof of the theorem for rational exponents is based on a simple observation whose proof nonetheless requires a few technical steps that are based on Proposition \ref{trace prop}, which itself relies on Lemma \ref{trace ineq 2} established in \cite{Bourin1999,Fujii2005}.

\begin{lemma}\label{trace ineq 2}
Let $F, G$ be real-valued functions on some domain $D\subseteq\mathbb{R}$ and $\Sigma$, $X$ be Hermitian matrices, such that the spectrum of $\Sigma$ is contained in $D$. If $(F,G)$ is an antimonotone pair so that $(F(a)-F(b))(G(a)-G(b))\leq 0$ for all $a,b\in D$, then
\begin{equation}
\tr\left[F(\Sigma)XG(\Sigma)X\right]\geq\tr\left[F(\Sigma)G(\Sigma)X^2\right].
\end{equation}
\end{lemma}

\begin{proposition}\label{trace prop}
If $\Sigma\in S^+_n$ and $X$ is a Hermitian matrix, then
\begin{equation} \label{trace prop eq}
\tr\left(\Sigma^{-2-k}X\Sigma^{k}X\right)\geq \tr\left(\Sigma^{-1-k}X\Sigma^{-1+k}X\right),
\end{equation}
for integers $k\geq 0$.
\end{proposition}

\begin{proof}
Define $F,G:(0,\infty)\rightarrow\mathbb{R}$ by $F(x):=x^{-1-2k}$ and $G(x):=x$, and note that $(F(a)-F(b))(G(a)-G(b))\leq 0$  for all $a,b> 0$. Let  $\Sigma\in S^+_n$ and $X$ be a Hermitian matrix. Then, we have
\begin{align}
\tr\left(\Sigma^{-2-k}X\Sigma^{k}X\right) &= \tr\left[\Sigma^{-1-2k}\left(\Sigma^{\frac{-1+k}{2}}X\Sigma^{\frac{-1+k}{2}}\right)\Sigma \left(\Sigma^{\frac{-1+k}{2}}X\Sigma^{\frac{-1+k}{2}}\right)\right] \nonumber \\
& \geq  \tr\left[\Sigma^{-2k}\left(\Sigma^{\frac{-1+k}{2}}X\Sigma^{\frac{-1+k}{2}}\right) \left(\Sigma^{\frac{-1+k}{2}}X\Sigma^{\frac{-1+k}{2}}\right)\right] \\
& = \tr\left(\Sigma^{-1-k}X\Sigma^{-1+k}X\right),
\end{align}
following an application of Lemma \ref{trace ineq 2} with the Hermitian matrix replaced by
$\Sigma^{\frac{k-1}{2}}X\Sigma^{\frac{k-1}{2}}$.
\end{proof}

\begin{proof}[Proof of Theorem \ref{GenLownerHeinz}]:
The differential $df_{1/p}\vert_{\Sigma}:T_{\Sigma}S^+_n\to T_{f_{1/p}(\Sigma)}S^+_n$ of $f_{1/p}$ satisfies the generalized Sylvester equation
\begin{equation} \label{gen syl}
\sum_{j=0}^{p-1}(\Sigma^{1/p})^{p-1-j}(df_{1/p}\vert_{\Sigma}X)(\Sigma^{1/p})^{j}=X,
\end{equation}
for every $X\in T_{\Sigma}S^+_n$. Thus,
\begin{equation}  \label{rootdiff}
\sum_{j=0}^{p-1}(\Sigma^{1/p})^{p-1-j-\frac{1}{2}p}(df_{1/p}\vert_{\Sigma}X)(\Sigma^{1/p})^{j-\frac{1}{2}p}=\Sigma^{-1/2}X\Sigma^{-1/2}.
\end{equation}
Taking the trace of (\ref{rootdiff}) yields
\begin{align}
\tr\left(\sum_{j=0}^{p-1}(\Sigma^{1/p})^{\frac{1}{2}p-1-j}(df_{1/p}\vert_{\Sigma}X)(\Sigma^{1/p})^{j-\frac{1}{2}p}\right) &= \tr(\Sigma^{-1/2}X\Sigma^{-1/2})  \\
\implies \quad \tr\left(\sum_{j=0}^{p-1}\Sigma^{-1/p}(df_{1/p}\vert_{\Sigma}X)\right) &= \tr(\Sigma^{-1}X).  \label{trace1}
\end{align}
That is,
$p\tr\left((f_{1/p}(\Sigma))^{-1}(df_{1/p}\vert_{\Sigma}X)\right) = \tr(\Sigma^{-1}X)$, for all $X\in T_{\Sigma}S^+_n$. Now taking the trace of the square of (\ref{rootdiff}), we obtain
\begin{equation} \label{rootdiff2}
\tr\left(\sum_{i,j=0}^{p-1}(\Sigma^{1/p})^{i-j-1}(df_{1/p}\vert_{\Sigma}X)(\Sigma^{1/p})^{j-i-1}(df_{1/p}\vert_{\Sigma}X)\right)=\tr(\Sigma^{-1}X\Sigma^{-1}X).
\end{equation}
The left-hand side of (\ref{rootdiff2}) can be rewritten as
\begin{align}
\sum_{i,j=0}^{p-1}&\tr\left[\left((\Sigma^{1/p})^{\frac{j-i-1}{2}}(df_{1/p}\vert_{\Sigma}X)(\Sigma^{1/p})^{\frac{j-i-1}{2}}\right)^2 \left((\Sigma^{1/p})^{i-j}\right)^2\right] \\
  & \geq \quad   \sum_{i,j=0}^{p-1}\tr\left[\left((\Sigma^{1/p})^{\frac{j-i-1}{2}}(df_{1/p}\vert_{\Sigma}X)(\Sigma^{1/p})^{\frac{j-i-1}{2}}(\Sigma^{1/p})^{i-j}\right)^2\right] \\
  & = \quad \sum_{i,j=0}^{p-1}\tr\left[\Sigma^{-1/p}(df_{1/p}\vert_{\Sigma}X)\Sigma^{-1/p}(df_{1/p}\vert_{\Sigma}X)\right] \\
  & = \quad p^2\tr\left[(f_{1/p}(\Sigma))^{-1}(df_{1/p}\vert_{\Sigma}X)(f_{1/p}(\Sigma))^{-1}(df_{1/p}\vert_{\Sigma}X)\right],
\end{align}
where the inequality follows from an application of Lemma \ref{inequalitylemma}. Thus, 
\begin{equation}
\tr\left[\left((f_{1/p}(\Sigma))^{-1}(df_{1/p}\vert_{\Sigma}X)\right)^2\right]\leq \frac{1}{p^2} \tr(\Sigma^{-1}X\Sigma^{-1}X).
\end{equation}
Combined with (\ref{trace1}), this implies that
\begin{align}
[\tr\left((f_{1/p}(\Sigma))^{-1}(df_{1/p}\vert_{\Sigma}X)\right)]^2-\mu\tr\left[\left((f_{1/p}(\Sigma))^{-1}(df_{1/p}\vert_{\Sigma}X)\right)^2\right] & \nonumber  \\
\geq  \frac{1}{p^2}\left([\tr(\Sigma^{-1}X)]^2-\mu\tr(\Sigma^{-1}X\Sigma^{-1}X)\right) \geq 0, &
\end{align}
for all $X\in\mathcal{K}(\Sigma)$. That is, $(df_{1/p}\vert_{\Sigma})\mathcal{K}(\Sigma)\subseteq\mathcal{K}(f_{1/p}(\Sigma))$ for any choice of $\mu$. 

This result can be extended to all rational powers $q/p\in\mathbb{Q}\cap[0,1]$ by combining two observations. First, since the inverse of the $p$-th root matrix function $f_{1/p}$ is the $p$-th power function $f_p:\Sigma\mapsto\Sigma^p$ and $f_{1/p}$ contracts the invariant cone field $\mathcal{K}$, $f_p$ must expand $\mathcal{K}$. Second, this expansion is greater for larger $p$. That is,
for positive integers $p_1\leq p_2$, 
\begin{equation}  \label{rational}
\left(d\tau_{\Sigma^{-1/2p_1}}\vert_{\Sigma^{p_1}}\circ df_{p_1}\vert_{\Sigma}\right)\mathcal{K}(\Sigma)\subseteq \left(d\tau_{\Sigma^{-1/2p_2}}\vert_{\Sigma^{p_2}}\circ df_{p_2}\vert_{\Sigma}\right)\mathcal{K}(\Sigma).
\end{equation}
Thus, the map $f_{q/p}=f_q\circ f_{1/p}$ is differentially positive, since the contraction of the cone field by $f_{1/p}$ will dominate the expansion of the cone field by $f_q$ for $p\geq q$. Note that the contractions and expansions referred to here need not be strict for the argument to hold. To prove (\ref{rational}), it is sufficient to show that the map $f_{p+1}$ expands the cone field at least as much as $f_{p}$ for any $p\in\mathbb{N}$. This is done by showing that
\begin{equation} \label{expansion}
df_p\vert_{\Sigma}X\in\partial\mathcal{K}(\Sigma^p) \implies df_{p+1}\vert_{\Sigma}X\notin\operatorname{int}\mathcal{K}(\Sigma^{p+1}),
\end{equation}
for any $\Sigma\in S^+_n$ and $X\in T_{\Sigma}S^+_n$, where $\partial\mathcal{K}(\Sigma^p)$ denotes the boundary of $\mathcal{K}(\Sigma^p)$.
Note that $df_p\vert_{\Sigma}X\in\partial\mathcal{K}(\Sigma^p)$ implies that $X\in\mathcal{K}(\Sigma)$, since $f_p$ expands $\mathcal{K}$. The implication in (\ref{expansion}) shows that the expansion of the cone field by $f_{p+1}$ is at least as great as that of $f_p$ by linearity of the differential maps. Using $\tr(f_{p}(\Sigma)^{-1}df_p\vert_{\Sigma}X)=p\tr(\Sigma^{-1}X)$, we see that 
$df_p\vert_{\Sigma}X\in\partial\mathcal{K}(\Sigma^p)$ is equivalent to 
\begin{equation} \label{expansion2}
p^2\tr(\Sigma^{-1}X)^2=\mu\sum_{i,j=0}^{p-1}\tr\left(\Sigma^{-1+i-j}X\Sigma^{-1+j-i}X\right).
\end{equation}
Assuming (\ref{expansion2}), we have
\begin{align} \label{expansion3}
&\left[\tr\left((f_{p+1}(\Sigma))^{-1}(df_{p+1}\vert_{\Sigma}X)\right)\right]^2-\mu\tr\left[\left((f_{p+1}(\Sigma))^{-1}(df_{p+1}\vert_{\Sigma}X)\right)^2\right]  \\
& = (p+1)^2\tr(\Sigma^{-1}X)^2-\mu\sum_{i,j=0}^{p}\tr\left(\Sigma^{-1+i-j}X\Sigma^{-1+j-i}X\right) \nonumber \\
& = \frac{\mu(p+1)^2}{p^2}\sum_{i,j=0}^{p-1}\tr\left(\Sigma^{-1+i-j}X\Sigma^{-1+j-i}X\right) - \mu\sum_{i,j=0}^{p}\tr\left(\Sigma^{-1+i-j}X\Sigma^{-1+j-i}X\right), \nonumber
\end{align}
where the last equation follows from substitution using (\ref{expansion2}). Using the simplification
$\sum_{i,j=0}^{p-1}\tr\left(\Sigma^{-1+i-j}X\Sigma^{-1+j-i}X\right) =\sum_{k=0}^{p-1}\alpha_k\tr\left(\Sigma^{-k-1}X\Sigma^{k-1}X\right)$, where $\alpha_0=p$ and $\alpha_k=2(p-k)$ for $k\geq 1$, (\ref{expansion3}) reduces to 
\begin{align}
\mu\left[\left(p\frac{(p+1)^2}{p^2}-(p+1)\right)\tr\left(\Sigma^{-1}X\Sigma^{-1}X\right) \right.  \nonumber \\
+\frac{(p+1)^2}{p^2}\sum_{k=1}^{p-1}2(p-k)&\tr\left(\Sigma^{-1-k}X\Sigma^{-1+k}X\right) \nonumber \\
-\left. \sum_{k=1}^p 2(p+1-k)\tr\left(\Sigma^{-1-k}X\Sigma^{-1+k}X\right)\right] 
\end{align}
\begin{align}
 = \mu&\left[\frac{p+1}{p}\tr\left(\Sigma^{-1}X\Sigma^{-1}X\right)  \right.  \nonumber \\
& + \sum_{k=1}^{p-1}\beta_k\tr\left(\Sigma^{-1-k}X\Sigma^{-1+k}X\right)
  \left. -2\tr\left(\Sigma^{-1-p}X\Sigma^{-1+p}X\right)\right], \label{expansion4}
\end{align}
where
\begin{equation}
\beta_k=2\frac{(p+1)^2(p-k)}{p^2}-2(p+1-k).
\end{equation}
We find that $\beta_k\geq 0$  if and only if $k\leq l := \lfloor p/2\rfloor$, where $\lfloor\cdot \rfloor$ identifies the integer part of its argument.
Thus, through repeated applications of Proposition \ref{trace prop}, we see that (\ref{expansion4}) is less than or equal to
\begin{align} \label{tech detail}
\mu\left(\frac{p+1}{p}+\sum_{k=1}^{l}\beta_k\right)\tr\left(\Sigma^{-1-l}X\Sigma^{-1+l}X\right) 
-\mu\left(2+\sum_{k=l+1}^{p-1}|\beta_k|\right)\tr\left(\Sigma^{-2-l}X\Sigma^{l}X\right) 
\nonumber \\
= \mu \left(2+\frac{(p-l-1)(l+2pl-p)}{p^2}\right)\left[\tr\left(\Sigma^{-1-l}X\Sigma^{-1+l}X\right)-\tr\left(\Sigma^{-2-l}X\Sigma^{l}X\right)\right],
\end{align}
which is nonpositive by a final application of Proposition \ref{trace prop}. This completes the proof of (\ref{expansion}).

Finally, we extend the result to all real exponents $r\in [0,1]$. Assume for a contradiction that there exists some $r\in(0,1)$ and $\Sigma_1,\Sigma_2\in S^+_n$ such that $\Sigma_1\geq \Sigma_2$ and $\Sigma_1^r < \Sigma_2^r$. Define 
$E=\{x\in (0,1): \Sigma_1^x < \Sigma_2^x\}$ and note that $E\neq \emptyset$ since $r\in E$. As $E$ is an open set in $\mathbb{R}$, there exists some $s\in\mathbb{Q}\cap E$ so that $\Sigma_1^s < \Sigma_2^s$, which is a contradiction. Therefore, $f_{r}$ is monotone for all $r\in [0,1]$ with respect to any of the affine-invariant orders parametrized by $\mu$.
\end{proof}

\begin{remark} The geometric insight provided by differential positivity clarifies the duality between the monotonicity of the function $f_r:\Sigma\mapsto\Sigma^{r}$ for $0 < r  < 1$ and its non-monotonicity for $r>1$, which may seem somewhat mysterious otherwise. Specifically, since the inverse of the function $f_r$ is given by $f_{1/r}$, we see that if $f_r$ contracts affine-invariant cone fields for $r\in (0,1)$ at every point, then $f_{1/r}$ must expand the same cone fields. Indeed, if the contraction of $\mathcal{K}$ by $f_{r}$ is strict at some $\Sigma\in S^+_n$, then 
$f_{1/r}$ cannot be differentially positive with respect to $\mathcal{K}$ and so is not monotone with respect to $\leq_{\mathcal{K}}$. See Figure \ref{fig:inverse}. 
To show that this is indeed the case for any of the affine-invariant cone fields (\ref{quadField}), we note that at any $\Sigma\in S^+_n$, $X_{\Sigma}=\Sigma\in T_{\Sigma}S^+_n$ lies in the interior of $\mathcal{K}(\Sigma)$, since $(\tr(\Sigma^{-1}X_{\Sigma}))^2-\mu\tr(\Sigma^{-1}X_{\Sigma}\Sigma^{-1}X_{\Sigma})=n^2-\mu n\ > 0$ and $\tr(\Sigma^{-1}X_{\Sigma})=\tr(I)=n > 0$ for $\mu\in (0,n)$. Let $\Sigma=\operatorname{diag}(\sigma_1,\sigma_2,\ldots,\sigma_n)$ be any diagonal matrix in $S^+_n$ with $\sigma_1>\sigma_2$. As $X_{\Sigma}=\Sigma\in\operatorname{int}\mathcal{K}(\Sigma)$, there exists some $\delta>0$ such that 
\begin{equation}
X=(x_{ij}) =
\begin{pmatrix}
\sigma_1 & \delta \cr
\delta & \sigma_2 \cr
& & \sigma_3 \cr
&&& \ddots \cr 
&&&&\sigma_n \cr
\end{pmatrix}
\end{equation}
lies on the boundary of $\partial \mathcal{K}(\Sigma)$. Specifically, we find that
\begin{align}
&\left(\tr(\Sigma^{-1}X)\right)^2-\mu\tr\left(\Sigma^{-1}X\Sigma^{-1}X\right)\\
&= \left(\sum_i\frac{x_{ii}}{\sigma_i}\right)^2 -\mu\left(\sum_i\frac{x^2_{ii}}{\sigma^2_i}+\frac{2}{\sigma_1\sigma_2}\delta^2\right) = n^2 - \mu\left(n+\frac{2}{\sigma_1\sigma_2}\delta^2\right)
\end{align}
 vanishes when
 \begin{equation}
 \delta^2 = \frac{n(n-\mu)\sigma_1\sigma_2}{2\mu}.
 \end{equation}
Now for this choice of $X$, the inequality (\ref{trace prop eq}) of Proposition \ref{trace prop} with $k=0$ becomes 
strict as
\begin{equation}
\tr\left(\Sigma^{-1}X\Sigma^{-1}X\right) 
=n+\frac{2}{\sigma_1\sigma_2}\delta^2 
< n + \left(\frac{1}{\sigma_1^{2}}+\frac{1}{\sigma_2^{2}}\right)\delta^2 = \tr(\Sigma^{-2}X^2),
\end{equation}
since $(1/\sigma_1-1/\sigma_2)^2 > 0$. As this inequality is used to derive (\ref{tech detail}), which is used to prove (\ref{rational}), it follows that the contraction of $\mathcal{K}$ by $f_r$ is strict at some $\Sigma\in S^+_n$ for $r\in (0,1)$. Therefore, $f_r$ cannot be monotone with respect to $\leq_{\mathcal{K}}$ for $r>1$. 
\end{remark}

\begin{figure}
\centering
\includegraphics[width=0.6\linewidth]{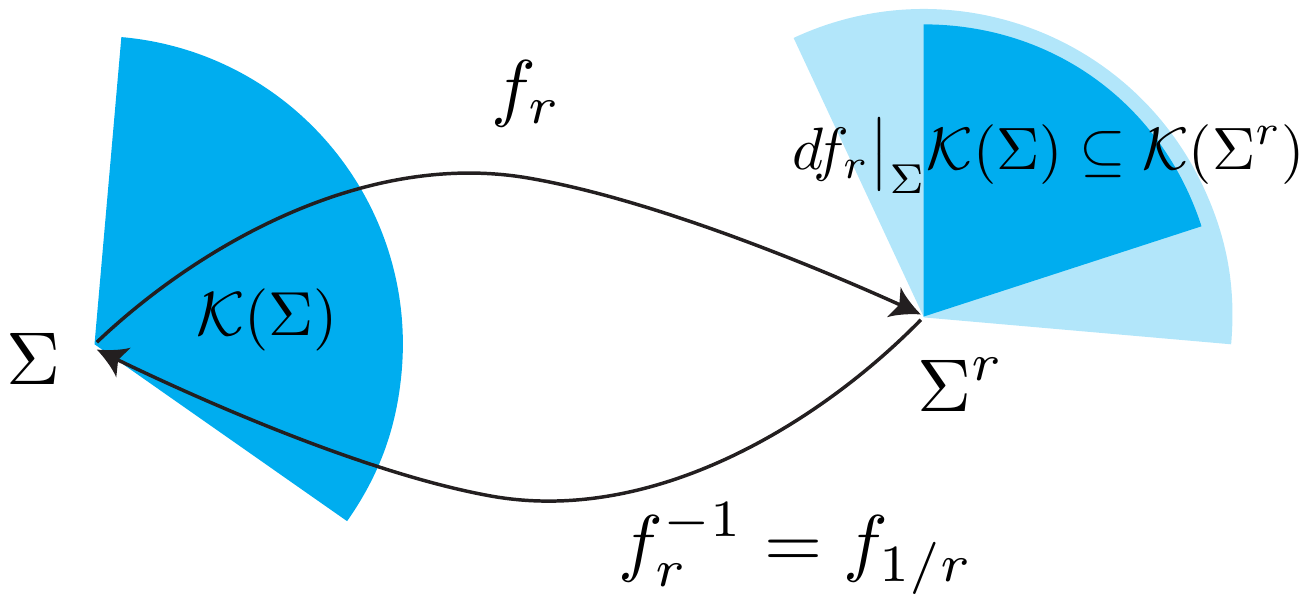}
  \caption{Contraction of affine-invariant cone fields by $f_r:\Sigma\mapsto\Sigma^r$ for $0 < r < 1$ corresponds to expansion of affine-invariant cone fields by the inverse map $f_r^{-1}=f_{1/r}:\Sigma\mapsto\Sigma^{1/r}$.
  }
  \label{fig:inverse}
\end{figure}

\subsection{Matrix inversion}

Consider the matrix inversion map $f(\Sigma)=\Sigma^{-1}$. The differential $df\vert_{\Sigma}:T_{\Sigma}S^+_n\to T_{\Sigma^{-1}}S^+_n$ of $f$ is given by 
\begin{equation}
df\vert_{\Sigma}X=-\Sigma^{-1}X\Sigma^{-1}.
\end{equation}
To show this, it is sufficient to consider the geodesic from $\Sigma$ in the
direction $X\in T_{\Sigma}S^+_n$ given by 
\begin{equation}
\gamma(t)=\Sigma^{1/2}\exp(t\Sigma^{-1/2}X\Sigma^{-1/2})\Sigma^{1/2},
\end{equation}
and note that $(f\circ\gamma)(t)=\Sigma^{-1/2}\exp(-t\Sigma^{-1/2}X\Sigma^{-1/2})\Sigma^{-1/2}$ so that 
\begin{equation}
(f\circ\gamma)'(0) = \Sigma^{-1/2}(-\Sigma^{-1/2}X\Sigma^{-1/2})e^{-t\Sigma^{1/2}X\Sigma^{1/2}}\Sigma^{-1/2}\big\vert_{t=0}=-\Sigma^{-1}X\Sigma^{-1}.
\end{equation}
Thus,
$
\tr(\Sigma\, (df\vert_{\Sigma}X)) =-\tr(\Sigma^{-1}X)$ and $
\tr\left[(\Sigma\, df\vert_{\Sigma}X)^2\right] = \tr(\Sigma^{-1}X\Sigma^{-1}X)$. Therefore, noting the conditions in (\ref{quadField}), it is clear that $\Sigma\mapsto \Sigma^{-1}$ reverses the ordering of positive definite matrices for any of the affine-invariant orders since
\begin{equation}
\tr((f(\Sigma))^{-1}(df\vert_{\Sigma}X)) =-\tr(\Sigma^{-1}X).
\end{equation}
That is, 
\begin{equation}
\Sigma_1 \geq_{\mathcal{K}} \Sigma_2 \quad \implies \quad \Sigma_2^{-1} \geq_{\mathcal{K}} \Sigma_1^{-1}, 
\end{equation}
for any of the affine-invariant cone fields $\mathcal{K}$ in (\ref{quadField}).

\subsection{Scaling and congruence transformations}

Consider the function $S_{\lambda}:S^+_n\to S^+_n$ defined by $S_{\lambda}(\Sigma)=\lambda \Sigma$, where $\lambda>0$ is a scalar. The differential $dS_{\lambda}\vert_{\Sigma}:T_{\Sigma}S^+_n\to T_{\lambda\Sigma}S^+_n$ is given by $dS_{\lambda}\vert_{\Sigma} X=\lambda X$. Substituting into the formula for the family of quadratic affine-invariant cones (\ref{quadField}), we find that
\begin{align}
&\left[\tr\left(S_{\lambda}(\Sigma)^{-1}(dS_{\lambda}\vert_{\Sigma}X)\right)\right]^2 - \mu \tr\left(S_{\lambda}(\Sigma)^{-1}(dS_{\lambda}\vert_{\Sigma}X) S_{\lambda}(\Sigma)^{-1}(dS_{\lambda}\vert_{\Sigma}X)\right) \nonumber \\
& = \left[\tr\left(\frac{1}{\lambda}\Sigma^{-1}\lambda X\right)\right]^2-\mu\tr\left(\frac{1}{\lambda}\Sigma^{-1}\lambda X\right)^2 = [\tr(\Sigma^{-1}X)]^2-\mu\tr(\Sigma^{-1}X)^2 \geq 0
\end{align}
for any $X\in \mathcal{K}(\Sigma)$. Thus, $S_{\lambda}$ is differentially positive and so preserves the affine-invariant orders induced by any of the cone fields (\ref{quadField}). This is of course a special case of a more general result about congruence transformations $\tau_{A}(\Sigma)=A\Sigma A^T$, where $A\in GL(n)$. Congruence transformations can be thought of as generalizations of scaling transformations on $S^+_n$. The preservation of affine-invariant orders by congruence transformations follows by construction. If $\Sigma_1 \leq_{\mathcal{K}} \Sigma_2$ for some partial order induced by an affine-invariant cone field $\mathcal{K}$, then there exists a conal curve $\gamma$ from $\Sigma_1$ to $\Sigma_2$. It follows from the definition of affine-invariant cone fields that congruence transformations map conal curves to conal curves in $S^+_n$. That is, $\tau_A(\gamma(t))$ is a conal curve joining $\tau_A(\Sigma_1)$ to $\tau_A(\Sigma_2)$.

\subsection{Translations}

It is important to note that translations do not generally preserve an affine-invariant order unless the associated affine-invariant cone field happens to also be translation invariant. 

\begin{proposition}
Let $\leq_{\mathcal{K}}$ denote the partial order induced by an affine-invariant cone field $\mathcal{K}$ on $S^+_n$. If $\mathcal{K}$ is not translation invariant, then there exists a translation $T_C:S^+_n\to S^+_n$, $T_C(\Sigma)=\Sigma+C$ that does not preserve $\leq_{\mathcal{K}}$.
\end{proposition}

\begin{proof}If $\mathcal{K}$ is not translation invariant, then there exist $\Sigma_1,\Sigma_2\in S^+_n$ such that $dT_{(\Sigma_2-\Sigma_1)}\vert_{\Sigma_1}\mathcal{K}(\Sigma_1)\neq\mathcal{K}(\Sigma_2)$, where $T_{(\Sigma_2-\Sigma_1)}(\Sigma)=\Sigma+(\Sigma_2-\Sigma_1)$. Thus there exists some $\delta\Sigma$ in the cone at either $\Sigma_1$ or $\Sigma_2$ that cannot be identified with an element of the cone at the other point under translation. Without loss of generality, assume that $\delta\Sigma\in \mathcal{K}(\Sigma_1)$ and $dT_{(\Sigma_2-\Sigma_1)}\big\vert_{\Sigma_1}(\delta\Sigma)\notin\mathcal{K}(\Sigma_2)$. For an affine-invariant cone field $\mathcal{K}$, we have 
\begin{equation}
\mathcal{K}(\lambda\Sigma)=d\tau_{\lambda^{1/2}I}\big\vert_{\Sigma}\mathcal{K}(\Sigma)=dS_{\lambda}\big\vert_{\Sigma}\mathcal{K}(\Sigma)=\lambda\mathcal{K}(\Sigma)=\mathcal{K}(\Sigma)
\end{equation}
for any $\lambda>0$ and $\Sigma\in S^+_n$. That is, the cone field is translationally invariant along each ray $\gamma(t)=t\Sigma$, $t>0$. Thus, we can identify $\mathcal{K}(\Sigma_2)$ through translation with any cone $\mathcal{K}(\lambda\Sigma_2)$ where $\lambda>0$. It follows that $dT_{(\lambda\Sigma_2-\Sigma_1)}\big\vert_{\Sigma_1}(\delta\Sigma) \notin\mathcal{K}(\lambda\Sigma_2)$ for any $\lambda>0$. For sufficiently large $\lambda>0$, $C:=\lambda\Sigma_2-\Sigma_1$ is a positive definite matrix. Therefore, $T_C:S^+_n\to S^+_n$ is not differentially positive with respect to $\mathcal{K}$ and hence is not monotone with respect to $\leq_{\mathcal{K}}$.
\end{proof}

\section{Invariant half-spaces}  \label{half spaces}

\subsection{An affine-invariant half-space preorder}

The $\operatorname{Ad}_{O(n)}$-invariant condition $\tr(X)\geq 0$ on $T_{I}S^+_n$ in (\ref{quad}) picks out a pointed cone from the double cone defined by the non-negativity of the quadratic form $(\tr(X))^2-\mu\tr(X^2)$. Indeed, $\tr(X)\geq 0$ defines a half-space in $T_{I}S^+_n$ bounded by the hyperplane $\tr(X)=0$ in $T_{I}S^+_n$. The affine-invariant extension of this hyperplane to all of $S^+_n$ yields a distribution of rank $\dim S^+_n-1 = n(n+1)/2-1$ on $S^+_n$ given by $\tr(\Sigma^{-1/2}X\Sigma^{-1/2})=\tr(\Sigma^{-1}X)=0$ for $X\in T_{\Sigma}S^+_n$. The corresponding affine-invariant half-space field $\mathcal{H}_{\Sigma}$ on the tangent bundle $TS^+_n$ simply takes the form
\begin{equation} \label{preorder}
\mathcal{H}_{\Sigma}=\{X\in T_{\Sigma}S^+_n:\tr(\Sigma^{-1}X)\geq 0\}.
\end{equation}
A half-space field of this form induces a partial preorder $\preceq_{\mathcal{H}}$ on $S^+_n$. That is, a binary relation that is reflexive and transitive. The antisymmetry condition required for a preorder to be a partial order does not hold since $\mathcal{H}_{\Sigma}$ is not a pointed cone. Nonetheless, one can ask whether any two given matrices $\Sigma_1,\Sigma_2\in S^+_n$ satisfy $\Sigma_1\preceq_{\mathcal{H}}\Sigma_2$, or if a given function on $S^+_n$ is monotone with respect to the preorder induced by (\ref{preorder}). The monotonicity of a function with respect to a preorder still gives geometric insight into the effects of the function on the space on which it acts and the discrete-time dynamics defined by its iterations.

To illustrate this we return to a puzzling aspect concerning the monotonicity of the function $f_r(x)=x^r$ on the real line for $r>0$ and its analogue result for positive semidefinite matrices. Namely, that the map $f_r$ is monotone on $S^+_n$ with respect to an affine-invariant partial order if $r\in [0,1]$ but is not monotone on $S^+_n$ for $r>1$. We will show that the monotonicity on the real line for $r>0$ is inherited in the matrix function setting in the form of a one-dimensional monotonicity expressed as the preservation of the affine-invariant half-space preorder for any $r>0$. 

\begin{proposition}\label{preorder prop}
The function $f_r:\Sigma \mapsto \Sigma^r$ is monotone on $S^+_n$ with respect to the affine-invariant half-space preorder $\preceq_{\mathcal{H}}$  for any $r>0$.
\end{proposition}

\begin{proof}
Let $p,q\in \mathbb{N}$ be positive integers. The map $f_{q/p}:\Sigma\mapsto \Sigma^{q/p}$ can be written as the composition $f_{1/p}\circ f_q$ with differential
\begin{equation}\label{rational composition}
df_{q/p}\vert_{\Sigma}=df_{1/p}\vert_{f_q(\Sigma)}\circ df_{q}\vert_{\Sigma}.
\end{equation}
Now since $df_q\vert_{\Sigma}$ is given by
\begin{equation}  
df_q\vert_{\Sigma}X=\sum_{j=0}^{q-1}\Sigma^{q-1-j}X\Sigma^j, \quad (X\in T_{\Sigma}S^+_n)
\end{equation}
and $df_{1/p}\vert_{\Sigma}$ is the unique solution of the generalized Sylvester equation (\ref{gen syl}), the differential $df_{q/p}\vert_{\Sigma}$ in (\ref{rational composition}) must satisfy
\begin{equation}
\sum_{i=0}^{p-1}(\Sigma^{q/p})^{p-1-i}(df_{q/p}\vert_{\Sigma}X)(\Sigma^{q/p})^i =
\sum_{j=0}^{q-1}\Sigma^{q-1-j}X\Sigma^j.
\end{equation}
Multiplying both sides of this equation by $\Sigma^{-q}$ and taking the trace of the resulting equation yields
\begin{align}
\tr\left(\sum_{i=0}^{p-1}(\Sigma^{q/p})^{-1-i}(df_{q/p}\vert_{\Sigma}X)(\Sigma^{q/p})^i\right) &=
\tr\left(\sum_{j=0}^{q-1}\Sigma^{-1-j}X\Sigma^{j}\right) \\
\implies \quad \tr\left(\sum_{i=0}^{p-1}\Sigma^{-q/p}(df_{q/p}\vert_{\Sigma}X)\right)&=\tr\left(\sum_{j=0}^{q-1}\Sigma^{-1}X\right) \\
\implies \quad p\tr\left(\Sigma^{-q/p}(df_{q/p}\vert_{\Sigma}X)\right) &= q \tr(\Sigma^{-1}X).
\end{align}
That is, $\tr\left((f_{q/p}(\Sigma))^{-1}df_{q/p}\vert_{\Sigma}X\right)=\frac{q}{p}\tr(\Sigma^{-1}X)$ for all $X\in T_{\Sigma}S^+_n$. A standard argument based on the density of positive rational numbers in the positive real line $\mathbb{R}_+$ gives 
\begin{equation}
\tr\left((f_{r}(\Sigma))^{-1}df_{r}\vert_{\Sigma}X\right)=r\tr(\Sigma^{-1}X)
\end{equation}
 for any real $r>0$. Therefore, we clearly have the implication
 \begin{equation}
 X\in \mathcal{H}_{\Sigma} \implies df_{r}\vert_{\Sigma}X\in \mathcal{H}_{f_r(\Sigma)}
 \end{equation}
 for all $X\in T_{\Sigma}S^+_n$, which is precisely the local characterization of the monotonicity of $f_r$ with respect to the preorder induced by $\mathcal{H}_{\Sigma}$.
\end{proof}
This result further highlights the natural connection between affine-invariance of causal structures on $S^+_n$ and monotonicity of the matrix power functions $f_r(\Sigma)=\Sigma^r$. In particular, $f_r$ is generally not monotone with respect to a preorder induced by a half-space field that is translation-invariant.

It should be noted that although the above proof has the virtue of being self-contained, Proposition \ref{preorder prop} can also be proven using results from Section \ref{partial orders}. Specifically, it should be clear from the material from that section that $\Sigma_1\preceq_{\mathcal{H}}\Sigma_2$ if and only if $\det\Sigma_1\preceq_{\mathcal{H}}\det\Sigma_2$, whence $f_r:\Sigma\mapsto\Sigma^r$ preserves $\preceq_{\mathcal{H}}$ precisely when
\begin{equation}
\det\Sigma_1\leq\det\Sigma_2 \implies \det f_r(\Sigma_1)\leq\det f_r(\Sigma_2).
\end{equation}
Since $\det f_r(\Sigma)=\det\Sigma^r=r(\det\Sigma)$, this is clearly the case for any $r>0$.

It is instructive to return to the $n=2$ case to obtain a visualization of the rank $2$ distribution $\mathcal{D}_{\Sigma}=\partial\mathcal{H}$ that defines the affine-invariant preorder induced by $\mathcal{H}_{\Sigma}$. As noted in Section \ref{visualization}, the set $S^+_2$ can be identified with the interior of the quadratic cone $K$ in $\mathbb{R}^3$ given by $z^2-x^2-y^2\geq 0$, $z\geq 0$ via a bijection $\phi:\Sigma\mapsto (x,y,z)$. 
At $\Sigma = \phi^{-1}(x,y,z)\in S^+_2$, the inequality $\operatorname{tr}(\Sigma^{-1}X)\geq 0$ takes the form $\ z\delta z - x \delta x - y \delta y \geq 0$, where $(\delta x, \delta y, \delta z)\in T_{(x,y,z)}K$ as shown in (\ref{halfspace}). The distribution $\partial\mathcal{H}$ that consists of the hyperplanes which form the boundary of the half-space field $\mathcal{H}_{\Sigma}$ are given by $\ z\delta z - x \delta x - y \delta y = 0$. This distribution is clearly integrable with integral submanifolds of the form
$
z^2-x^2 - y^2 = C$,
where $C\geq 0$ is a constant for each of the integral submanifolds, which form hyperboloids of revolution as shown in Figure \ref{fig:3}. As expected, these surfaces coincide with the submanifolds of constant determinant predicted in Section \ref{partial orders}.

\begin{figure}
\centering
\includegraphics[width=0.8\linewidth]{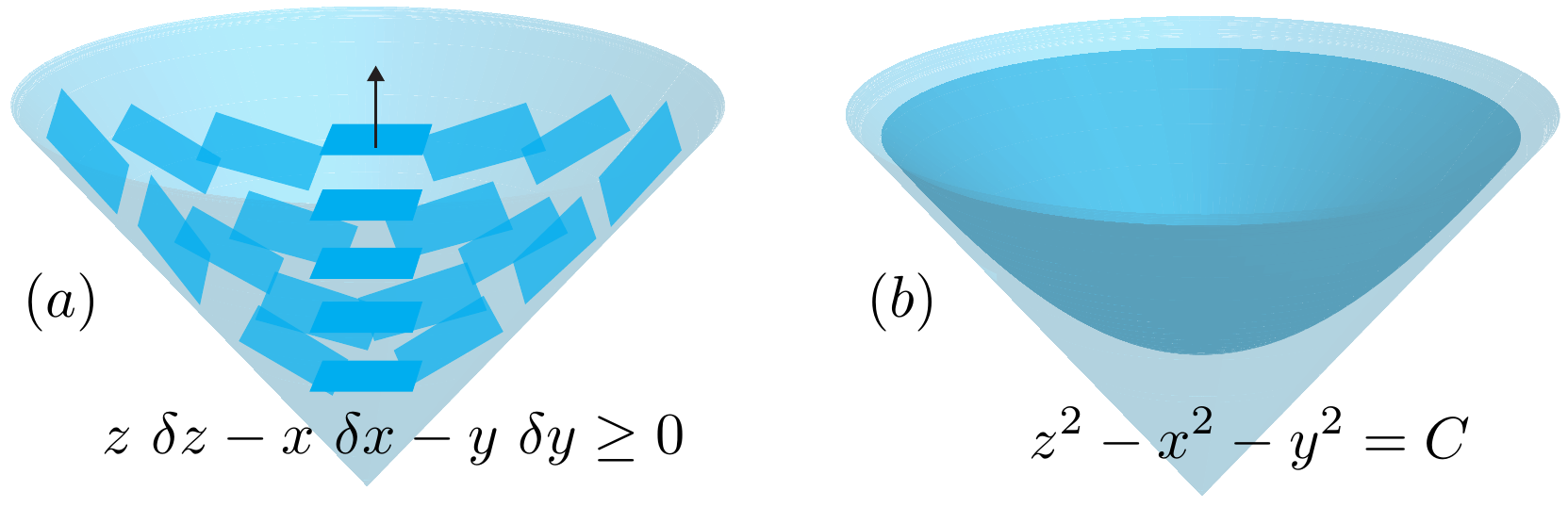}
  \caption{$(a)$ An illustration of the affine-invariant hyperplanes $\partial\mathcal{H}$ corresponding to $\tr(\Sigma^{-1}X) = 0$ against the backdrop of the cone $K=\{(x,y,z)\in \mathbb{R}^3:z^2-x^2-y^2\geq 0, z\geq 0\}$ identified with $S^+_2$. $(b)$ The distributions integrate to give a family of hyperboloids of revolution parametrized by $C > 0$.  The limiting case $C=0$ yields the boundary of the cone $K$. 
   }
  \label{fig:3}
\end{figure}

\subsection{The Toda and QR flows}

The \emph{Toda} flow is a well-know Hamiltonian dynamical system on the space of real symmetric matrices of fixed dimension $n$, which can be expressed in the Lax pair form
\begin{equation} \label{Lax}
\dot{X}(t)=[X,\pi_s(X)] = X\pi_s(X)-\pi_s(X)X,
\end{equation}
where $\pi_s(X)$ is the skew-symmetric matrix
$\pi_s(X)=
X_{ij}$ if $i>j$,   
$\pi_s(X)=0$ if $i=j$, and
$\pi_s(X)=-X_{ji}$ if $i<j$.
The QR-flow is a related dynamical system on $S^+_n$ that has close connections to the QR algorithm and is given by
\begin{equation}
\dot{\Sigma}(t)=[\Sigma,\pi_s(\log\Sigma)].
\end{equation}
The Lax pair formulations of the Toda and QR-flows show that these flows are \emph{isospectral}. That is, the eigenvalues of $X(t)$ and $\Sigma(t)$ are independent of $t$. Isospectral flows clearly preserve all translation invariant orders that possess spectral characterizations.

In \cite{Lagarias1991}, the following theorem is established for the projected Toda and QR flows. The projected flows refer to projections of the flows to the $r\times r$ upper left corner principal submatrices of $X(t)$ and $\Sigma(t)$, i.e., the flows of $X_r(t)=E_r^TX(t)E_r$ and $\Sigma_r(t)=E_r^T\Sigma(t)E_r$, where $E_r^T=
[I_r \; 0]$.

\begin{theorem}
For $1\leq r \leq n$ and any symmetric matrix $X(0)$ and symmetric positive definite matrix $\Sigma(0)$, the ordered eigenvalues of the projected Toda flow orbit $X_r(t)=E_r^TX(t)E_r$ and the projected QR flow orbit $\Sigma_r(t)=E_r^T\Sigma(t)E_r$ are nondecreasing functions of $t$.
\end{theorem}

\begin{corollary}
Let $f(x)$ be any nondecreasing real-valued function and $\alpha>0$. Then $F(t)=\tr(f(E_r^TX(t)E_r))$ and $G(t)=\tr(f(E_r^T\Sigma(t)^{\alpha}E_r))$ are nondecreasing functions of $t$ for $t\in \mathbb{R}$.
\end{corollary}

The geometric interpretation of the above corollary is that the generalized projected Toda and QR flows, $f(X_r(t))$ and $f(\Sigma_r(t))$, respectively, preserve the half-space preorder induced by the translation invariant half-space $\tr(X)\geq 0$. This is clear by noting that if $X(0),\hat{X}(0)$ are symmetric matrices such that $\tr(X(0)-\hat{X}(0))\geq 0$, then
\begin{equation}
\tr(f(X_r(t))-f(\hat{X}_r(t))) \geq \tr(X(0)-\hat{X}(0)) \geq 0, \quad \forall t>0,
\end{equation}
and similarly for the generalized projected QR flow.

\section{Matrix means}   \label{matrix means}

Notions of means and averaging operations on matrices are of great interest in matrix analysis and operator theory with numerous applications to fields such as radar data processing, medical imaging, statistics and machine learning. Adapting basic properties of means on the positive real line to the setting of positive definite matrices, we may define a matrix mean to be a continuous map $M:S^+_n\times S^+_n\rightarrow S^+_n$ that satisfies the following properties \\

\begin{enumerate}
\item  $M(\Sigma_1,\Sigma_2)=M(\Sigma_2,\Sigma_1)$ 
\item $\Sigma_1 \leq \Sigma_2  \implies \Sigma_1 \leq M(\Sigma_1,\Sigma_2) \leq \Sigma_2$
\item $M(A^T\Sigma_1 A, A^T\Sigma_2 A) = A^T M(\Sigma_1,\Sigma_2)A$, for all $A\in GL(n).$
\item $M(\Sigma_1,\Sigma_2)$ is monotone in $\Sigma_1$ and $\Sigma_2$. \\
\end{enumerate} 
In the existing literature on matrix means, the partial order $\leq$ in the above definition refers to the L\"owner order $\leq_{L}$. It is a nontrivial question whether a given map $M:S^+_n\times S^+_n\rightarrow S^+_n$ defines a matrix mean with respect to any of the new partial orders considered in this paper. A particularly important matrix mean that has been the subject of considerable interest in recent years is the geometric mean $M(\Sigma_1,\Sigma_2)=\Sigma_1\#\Sigma_2$ defined by
\begin{equation}  \label{geometric mean}
\Sigma_1\#\Sigma_2 = \Sigma_1^{1/2}\left(\Sigma_1^{-1/2}\Sigma_2\Sigma_1^{-1/2}\right)^{1/2}\Sigma_1^{1/2}.
\end{equation}
The following theorem shows that the geometric mean and affine-invariant orders on $S^+_n$ are intimately connected.
\begin{theorem}
The geometric mean $\#$ (\ref{geometric mean}) defines a matrix mean for any affine-invariant order $\leq$ on $S^+_n$.
\end{theorem}

\begin{proof}
The geometric mean $\Sigma_1\#\Sigma_2$ of two points $\Sigma_1,\Sigma_2\in S^+_n$ is the midpoint of the geodesic joining $\Sigma_1$ and $\Sigma_2$ in $S^+_n$ endowed with the standard Riemannian metric $ds^2=\tr[(\Sigma^{-1}d\Sigma)^2]$ \cite{Bhatia}. This geometric interpretation immediately implies $\Sigma_1\#\Sigma_2=\Sigma_2\#\Sigma_1$. Furthermore, given any affine-invariant order $\leq_{\mathcal{K}}$ induced by an affine-invariant cone field $\mathcal{K}$ and a pair of matrices satisfying $\Sigma_1\leq_{\mathcal{K}}\Sigma_2$, the geodesic $\gamma:[0,1]\to S^+_n$ from $\Sigma_1$ to $\Sigma_2$ is a conal curve by Theorem \ref{geodesic conal curves thm}. Hence, the midpoint $\Sigma_1\#\Sigma_2$ of $\gamma$ clearly satisfies $\Sigma_1 \leq_{\mathcal{K}} \Sigma_1\#\Sigma_2 \leq_{\mathcal{K}} \Sigma_2$.
Since congruence transformations are isometries, for any $A\in GL(n)$ the geodesic connecting $A^T\Sigma_1 A$ to $A^T\Sigma_2 A$ is given by $\tilde{\gamma}(t)=A^T\gamma(t) A$. Thus, $(A^T\Sigma_1A)\#(A^T\Sigma_2A)=A^T( \Sigma_1\#\Sigma_2)A$. Finally, for fixed $\Sigma_1\in S^+_n$, the function $F(\Sigma)=\Sigma_1\#\Sigma$ is monotone with respect to any affine-invariant order since congruence transformations preserve affine-invariant orders and the function $\Sigma\mapsto\Sigma^{1/2}$ is monotone for any affine-invariant order. By symmetry, $\#$ is also monotone with respect to its first argument. That is, the four conditions that define a matrix mean are all satisfied by the geometric mean for any choice of affine-invariant order.
\end{proof}

\section{Conclusion}

The choice of partial order is a key part of studying monotonicity of functions that is often taken for granted. Invariant cone fields provide a geometric approach to systematically construct `natural' orders by connecting the geometry of the state space to the search for orders. Coupled with differential positivity, invariant cone fields provide an insightful and powerful method for studying monotonicity, as shown in the case of $S^+_n$. Future work can focus on exploring the applications of the new partial orders presented in this paper to the study of dynamical systems and convergence analysis of algorithms defined on matrices. It may also be fruitful to explore the implications of this work in convexity theory. New notions of partial orders mean new notions of convexity. In this context it may be natural to consider the concept of geodesic convexity on $S^+_n$ with respect to the Riemannian structure on $S^+_n$, as well as the usual notion of convexity on sets of matrices that is based on translational geometry.

\bibliographystyle{siamplain}
\bibliography{references}

\end{document}